\documentclass[12pt]{article}
\usepackage{amsmath,amssymb,amsthm,geometry,mathtools}
\usepackage[noadjust]{cite}
\usepackage[pdfstartview=FitH,
CJKbookmarks=true,
bookmarksnumbered=true,
bookmarksopen=true,
colorlinks,
linkcolor=blue,
anchorcolor=blue,
citecolor=blue,
urlcolor=blue
]{hyperref}

\newcommand\myvskip{\vskip 6pt}

\title{Dualities and endomorphisms of pseudo-cones}

\author{Yun Xu\footnote{Email: xuyunll@163.com.},
Jin Li\footnote{Email: li.jin.math@outlook.com.},  
Gangsong Leng\footnote{Email: lenggangsong@163.com.}
\\
Department of Mathematics\\ Shanghai University\\
200444 S\lowercase{hanghai},  C\lowercase{hina}}

\date{}

\newtheorem{Lem}{Lemma}[section]
\newtheorem{Thm}[Lem]{Theorem}
\newtheorem{Rem}[Lem]{Remark}
\newtheorem{Cor}[Lem]{Corollary}

\DeclareMathOperator{\cl}{cl}
\DeclareMathOperator{\relint}{relint}

\DeclareMathOperator{\rec}{rec}
\DeclareMathOperator{\interior}{int}
\DeclareMathOperator{\bd}{bd}
\DeclareMathOperator{\lin}{lin}
\DeclareMathOperator{\aff}{aff}

\DeclareMathOperator{\R}{\mathbb{R}}
\newcommand\PC{\mathcal{PC}^n}
\DeclareMathOperator{\ML}{\mathcal{L}}

\begin{document}
\maketitle
\begin{quote}
\noindent \textbf{Abstract} \quad 
In this paper we study a class of convex sets which are called closed pseudo-cones and study a new duality of this class. 
It turns out that the duality characterizes closed pseudo-cones and is essentially the only possible abstract duality of them.
The characterization of the duality is corresponding to the classification of endomorphisms closed pseudo-cones.
\vspace{1em}
		
		
\end{quote}

\section{Introduction}\label{Sec:Intr}
Dualities of convex sets or convex functions play important roles in convex geometry. For example, polar sets, Legendre transforms, etc.
An interesting topic is studying  dualities and endomorphisms (correspondingly) by some of their abstract properties like involution, order-reversing and interchanging intersections and ``unions"; see Gruber \cite{MR1182848}, B\"{o}r\"{o}czky, Schneider \cite{MR2438994}, and Artstein-Avidan, V. Milman \cite{AM09Leg,MR2342714,MR2406688,AM11hid,MR2563222,MR2464254} and many others \cite{MR3439685,MR1138277,MR4252711,MR2985126,MR2468074,MR2795422,MR4175764,2110.11308,MR2238926,MR3062743,MR353147,2207.09758}.
See also studies of isometries \cite{LM2021,MR533134} and valuations \cite{Lud06,Lud10b}.
Recently, E. Milman, V. Milman and Rotem \cite{MR4175764} introduced a new duality of a subclass of compact convex sets, which are called reciprocal bodies. 

It seems that all previous studies of dualities of convex sets need to assume that the origin is contained in convex sets.
We will explain later why the polar operator is not a good duality if the origin is outside of convex sets.
This study not only just gives an example of duality of so-called pseudo-cones, but also establishes the complete classification of dualites of them.

We work in Euclidean space~$\R^n$.
Let $K$~be a nonempty convex set which does not contain the origin. It is called a convex \textit{pseudo-cone} if $\lambda x\in K$ for all $x\in K$ and $\lambda\geq 1$.
A convex pseudo-cone is called briefly a pseudo-cone in this paper if there is no confusion.
This new concept is closely related to $C$-close sets introduced by Schneider \cite{MR3810252}; see also \cite{YYZ2021+}.
Let $C$~be a closed convex cone in $\R^n$; i.e., $C$ is a closed convex set such that $\lambda x\in C$ for all $x\in C$ and $\lambda\ge 0$.  Further assume $C$ is pointed and has no empty interior. A closed convex set~$K\subset C$ is called \textit{$C$-close} if $C\setminus K$ is of positive finite volume (Lebesgue measure). We can verify that all $C$-close sets are closed pseudo-cones; see Theorem \ref{Thm:CCloseIsPseducone}. 
However, there is a closed pseudo-cone of full dimension but not $C$-close for any pointed closed convex cone~$C$. For instance, $K := \bigl\{(x,y)\in\R^2:xy\geq 1\text{ and }x,y>0\bigr\}$ in~$\R^2$.

Before introducing a new duality of pseudo-cones, it is helpful to remind that the polar set~$K^\circ$ of a closed convex set~$K$ is defined as
\begin{equation*}
K^\circ=\bigl\{x\in\R^n:\langle x,y\rangle\leq 1 \text{ for all } y\in K\big\}
\end{equation*}
and $C^\circ$ of a closed convex cone~$C$ can be defined equivalently to
\begin{equation*}
    C^\circ=\bigl\{x\in\R^n:\langle x, y\rangle\leq 0 \text{ for all } y\in C\bigr\}.
\end{equation*}
Now let $K$~be a pseudo-cone. The new duality $K^\ast$ of~$K$ is defined by
\begin{equation*}
    K^\ast=\bigl\{x\in\R^n:\langle x,y\rangle\leq -1 \text{ for all } y\in K\bigr\}.
\end{equation*}
The operator~``$\ast$" can, of course, be defined on all nonempty closed convex sets. In fact, it in return characterizes the class of all closed pseudo-cones since $K^\ast$~is either empty or a pseudo-cone for any nonempty closed convex set~$K$; see Lemma \ref{Lem:A*IsPC}.

The class of all closed pseudo-cones is a partially ordered set with the set inclusion.
Therefore, it is not hard to see that
\begin{equation*}
    \PC:=\{K\subset\R^n:K \text{ is a closed pseudo-cone, or }K=\emptyset, \text{ or }K=\R^n\}
\end{equation*}
is a lattice containing all closed pseudo-cones with
\begin{equation*}
    K\land L=K\cap L
\end{equation*}
and
\begin{equation*}
	K\lor L = \begin{cases}
		\cl[K,L]  & \text{ if }o\notin \cl[K,L];\\
		\R^n  & \text{ if }o\in \cl[K,L]
	\end{cases}
\end{equation*}
for any $K,L \in \PC$, see Lemma \ref{Lem:PC^nIsLattice}.
Here $\cl[K,L]$~is the closure of the convex hull of~$K$, $L$ and $o$~is the origin.

We extend ``$*$'' to~$\PC$ by defining $\emptyset^\ast = \R^n$ and $(\R^n)^\ast = \emptyset$.
Observe that the operator~``$\ast$" satisfies \\
(i) $K^{\ast\ast}=K$, \\
(ii) $K\subset L$ implies $K^\ast\supset L^\ast$, \\
(iii) $(K\lor L)^\ast = K^\ast\land L^\ast$, \\
(iv) $(K\land L)^\ast = K^\ast\lor L^\ast$, \\
\noindent whenever $K,L\in\PC$; see Section \ref{Sec:ProOfPC}.
Those properties also holds for ``$\circ$" on closed convex sets containing $o$.
However, ``$\circ$" does not satisfies (i), (iii) or (iv) on closed pseudo-cones. For example, let $n=2$, $K=\bigl\{(x,y):x\geq 0,\,y\geq 1\bigr\}$ and $L=\bigl\{(x,y):x\geq 0,\,y\leq -1\bigr\}$.

A operator satisfing (i) and~(ii) is called abstract duality in Artstein-Avidan and V. Milman \cite{MR2342714}.
Enlightened by \cite[Lemma 2]{AM09Leg}, the properties (iii) and~(iv) are deduced from the properties (i) and~(ii); see Lemma \ref{Lem:DualForLattice}.
One natural question is whether the new duality is essentially the only transform satisfying properties (i) and~(ii) or even only properties (iii) and~(iv).
The answers are the following.

\begin{Thm}\label{Thm:MainThm_Dual}
	Let $n\geq2$. A transform $\tau\colon\PC \to \PC$ satisfies
	\begin{gather*}
        \tau(K\land L)=\tau(K)\lor\tau(L),\\
	    \tau(K\lor L)=\tau(K)\land\tau(L)
	\end{gather*}
	for all~$K$, $L\in\PC$, if and only if either $\tau$ is constant, or $\tau(K)=g(K^\ast)$ for some $g\in\mathrm{GL}(n)$.
\end{Thm}

\begin{Cor}\label{Cor:MainCor}
	Let $n\geq2$. A transform $\tau\colon\PC \to \PC$ satisfies
	\begin{gather*}
		\tau \bigl(\tau (K)\bigr)=K,\label{equ:Cor:MainCor1}\\
		K \subset L \text{ implies } \tau (K) \supset
		\tau(L)\label{equ:Cor:MainCor2}
	\end{gather*}
	for all~$K$, $L\in\PC$, if and only if $\tau(K)=g(K^\ast)$ for some $g\in\mathrm{GL}(n)$.
\end{Cor}

Theorem \ref{Thm:MainThm_Dual} is equivalent to the following classification of endomorphisms.
\begin{Thm}
\label{Thm:MainThm_Identity}
	Let $n\geq2$. A transform $\varphi\colon\PC \to \PC$ satisfies
	\begin{gather*}
		\varphi(K\land L)=\varphi(K)\land\varphi (L),\label{equ:Thm:MainThm_Identity1}\\
		\varphi(K\lor L)=\varphi (K)\lor\varphi (L)\label{equ:Thm:MainThm_Identity2}
	\end{gather*}
	for all~$K$, $L\in\PC$, if and only if either $\varphi$ is constant, or $\varphi(K)=g(K)$ for some $g\in\mathrm{GL}(n)$.
\end{Thm}

Classifications of endomorphisms of closed convex cones were established by Schneider \cite{MR2468074} which required $n \ge 3$ since there is an additional endomorphism 
in the case $n=2$. 
But it is not a problem for pseudo-cones.
Classifications of endomorphisms of other closed convex sets were first established by Gruber \cite{MR1138277} on compact convex sets, and later \cite{MR1182848} on origin-symmetric compact convex sets containing the origin in interiors; by B\"{o}r\"{o}czky and Schneider \cite{MR2438994} on compact convex sets containing the origin in interiors; and by Slomka \cite{MR2795422} on all closed convex sets and on closed convex sets containing the origin. 
Although there are some common ideas between those previous results and this study, this study has some specific techniques and we do not think there is some easy method to obtain our results from those previous results or vice versa.

When this paper is almost completed, we found that Artstein-Avidan, Sadovsky, and Wyczesany \cite{2110.11308} had mentioned that ``$\ast$" is an order reversing quasi involution, i.e., it satisfies the property (ii) and $K \subset K^{\ast\ast}$ for any $K,L \subset \R^n$. According to properties of general order reversing quasi involutions established in the same paper, ``$\ast$" is an abstract duality of $\PC$ and $\PC$ is a lattice.
However, showing that ``$\ast$" is essentially the only abstract duality of $\PC$ and classifying endomorphisms of $\PC$ are completely new.
	
The paper is organised as follows. Preliminaries and notation are given in section \ref{Sec:PreNot}. 
In section \ref{Sec:ProOfPC}, we study some properties of pseudo-cones and the new duality.
The endomorphisms on~$\PC$ are classified in Section \ref{Sec:Endo} including the case $n=1$. 


	
\section{Preliminaries and notation}\label{Sec:PreNot}
Denote by $\R$ and $\R_+$, the set of real numbers and the set of non-negative real numbers, respectively. 
For $x,y \in \R^n$, we write $\langle x,y \rangle$ for the standard inner product, and $|x|$ for the Euclidean norm. 
The unit sphere is $\mathbb{S}^{n-1}:=\{x\in\R^n:|x|=1\}$. 
For $x \in \R^n$ and $\varepsilon>0$, a closed ball $B(x,\varepsilon):=\{y\in\R^n:|x-y|\leq\varepsilon\}$.

Let $H_{u,\alpha}=\{x\in\R^n:\langle x,u\rangle=\alpha\}$.
The two closed half-spaces separated by hyperplane $H_{u,\alpha}$ are $H_{u,\alpha}^+:=\{x\in\R^n:\langle x,u\rangle\geq\alpha\}$, and
$H_{u,\alpha}^-:=\{x\in\R^n:\langle x,u\rangle\leq\alpha\}$.
A $k$-flat is a $k$-dimensional affine subspace in $\R^n$.

The notation $\interior K$, $\cl K$, $\bd K$, $\relint K$ are the interior, closure, boundary and relative interior of $K\subset\R^n$, respectively.

For $K,L\subset\R^n$, $\Lambda\subset\R$, $\alpha\in\R$, we write $K+L=\{x+y:x\in K,y\in L\}$, $\Lambda K=\{\lambda x:\lambda\in\Lambda, x\in K\}$, $\langle K, L\rangle=\{\langle x, y\rangle:x\in K,y\in L\}$, and define $\Lambda\leq(\geq)\alpha$ if $\lambda\leq(\geq)\alpha$ for all $\lambda\in\Lambda$.

For every $x\neq o$, write $R_x=(0,+\infty)\{x\}$ and $\overline{x}=[1,+\infty)\{x\}$.
We further set $\overline o=\R^n$. 
In this paper, we do not distinguish between a singleton and the element of it, i.e. $x$ maybe mean $\{x\}$, if there is no confusion.
	
A set $K\subset\R^n$ is called convex if $(1-\lambda)x+\lambda y\in K$ whenever $x,y\in K$ and $\lambda\in[0,1]$. 
For $K,L\subset\R^n$, we denote by 
$[K,L]$ the convex hull of $K \cup L$, i.e. the smallest convex set containing $K \cup L$. In particular, for $x_1,x_2,\ldots,x_m\in\R^n$, $[x_1,x_2,\ldots,x_m]$ denotes the convex hull of $\{x_1,x_2,\ldots,x_m\}$.

If $K\subset\R^n$ is a nonempty closed convex set,
\begin{equation*}
\rec K:=\{x\in\R^n:K+x\subset K\}
\end{equation*}
is called the recession cone of $K$. In fact, $\rec K$ is always a closed convex cone.
	
If $K\subset\R^n$ is a closed convex set, $x\notin K$, then $p_K(x)$ denotes the metric projection point of $x$ onto $K$, i.e. $p_K(x)$ is the nearest point of $K$ from $x$. The support function $h_K\colon\R^n\rightarrow(-\infty,+\infty]$ of $K$ is defined by
$$h_K(u)=\sup\{\langle u,x\rangle:x\in K\}.$$
If $K$ is a closed pseudo-cone,
we define the radial function $\rho_K\colon (\R_+K)\setminus\{o\} \to \R$ of $K$ by
\begin{equation*}
    \rho_K(x)=-\min\{\lambda\in\R_+:\lambda x\in K\}.
\end{equation*}
Notice that $-\infty<\rho_K(x)<0$ for all $x\in (\R_+K)\backslash\{o\}$, since $K$ is closed and $o\notin K$.


A partially ordered set $(\mathcal{Z},\preceq)$ is a set $\mathcal{Z}$ together with a partial order $\preceq$, which satisfies that
	
(i) $z\preceq z$,
	
(ii) $z_1\preceq z_2\text{ and }z_2\preceq z_1\Rightarrow z_1=z_2$,
	
(iii) $z_1\preceq z_2\text{ and }z_2\preceq z_3\Rightarrow z_1\preceq z_3$.
\\whenever $z,z_1,z_2,z_3 \in \mathcal{Z}$.

Let $\mathcal{A}\subset\mathcal{Z}$ be a nonempty subset. 
The supremum of $\mathcal{A}$ (writing as $\sup\mathcal{A}$) is the least upper bound of $\mathcal{A}$; the infimum of $\mathcal{A}$ (writing as $\inf\mathcal{A}$) is the greatest lower bound of $\mathcal{A}$. See \cite[Chapter 2]{MR1940513} for details.
	
A lattice $\ML$ is a partial order set $(\ML,\preceq)$ such that both $\sup\{x,y\}$ and $\inf\{x,y\}$ exist in $\ML$ for any 2-element set $\{x,y\}\subset\ML$. Write
\begin{gather}\label{eq:03-12-1014}
    x\land y=\inf\{x,y\}\text{ and }x\lor y=\sup\{x,y\}.
\end{gather}
An endomorphism on a lattice $\ML$ is a map $\varphi:\ML\rightarrow \ML$ such that
\begin{gather*}
\varphi(x\land y)=\varphi(x)\land \varphi(y)\text{ and }\varphi(x\lor y)=\varphi(x)\lor \varphi(y).
\end{gather*}

The following lemma will be used many times implicitly.

\begin{Lem}\label{Lem:EndoIsInc}
    Let $(\ML,\lor,\land)$ be a lattice. If $\tau$ is an endomorphism of $\ML$, then
    \begin{gather}\label{eq:03-12-1016}
        x\preceq y \Longrightarrow \tau(x)\preceq \tau(y).
    \end{gather}
\end{Lem}
\begin{proof}
    If $x\preceq y$, by \eqref{eq:03-12-1014}, then $x\land y = x$, which implies
    \begin{gather*}
        \tau(x)\land \tau(y)=\tau(x\land y)=\tau(x).
    \end{gather*}
    Together with \eqref{eq:03-12-1014}, one gets
    \begin{gather*}
        \tau(x)\preceq \tau(y),
    \end{gather*}
    which deduces \eqref{eq:03-12-1016}. It completes the proof.
\end{proof}

The following lemma is analogous to \cite[Lemma 2]{AM09Leg} and shows that Corollary \ref{Cor:MainCor} is deduced from Theorem \ref{Thm:MainThm_Dual}.
	
\begin{Lem}\label{Lem:DualForLattice}
	Let $(\ML,\lor,\land)$ be a lattice, and let $\tau\colon\ML \to \ML$ be a bijection. The following properties are equivalent \\
	(i) $$x\preceq y \Leftrightarrow \tau(y)\preceq \tau(x) \text{ for any } x,y \in \ML,$$
	(ii) $$\tau(x\lor y)=\tau(x)\land \tau(y) \text{ for any } x,y \in \ML,$$
	(iii) $$\tau(x\land y)=\tau(x)\lor \tau(y) \text{ for any } x,y \in \ML.$$
\end{Lem}
\begin{proof}
(i) $\Rightarrow$ (ii). Observe that $\tau^{-1}$ also satisfies (i).
Since $a\land b\preceq a, b$, 
we have
$ \tau^{-1}(a),\tau^{-1}(b) \preceq \tau^{-1}(a\land b)$,
which implies 
$$ \tau^{-1}(a)\lor \tau^{-1}(b) \preceq \tau^{-1}(a\land b)$$ 
for any $a,b \in \ML$. 
Thus,
	$$\tau^{-1}(\tau(x\lor y))=x\lor y=\tau^{-1}(\tau(x))\lor \tau^{-1}(\tau(y))\preceq \tau^{-1}(\tau(x)\land \tau(y)),$$
	which implies $\tau(x)\land \tau(y)\preceq \tau(x\lor y)$.
	On the other hand, from $x,y\preceq x\lor y$, we have $\tau(x\lor y)\preceq \tau(x),\tau(y)$, which deduces $\tau(x\lor y)\preceq \tau(x)\land \tau(y)$. 
	Therefore, $$\tau(x\lor y)=\tau(x)\land \tau(y).$$
		
(ii) $\Rightarrow$ (i). If $x\preceq y$, then $\tau(y)=\tau(x\lor y)=\tau(x)\land \tau(y)$ which implies $\tau(y)\preceq \tau(x)$.
    If $\tau(y)\preceq \tau(x)$, then $\tau(y)=\tau(x)\land \tau(y)=\tau(x\lor y)$. 
    Since $\tau$ is a bijection, it turns out $y=x\lor y$ which deduces $x\preceq y$. 
    
    (i) $\Leftrightarrow$ (iii).	It is similar to the proof of ``(i) $\Leftrightarrow$ (ii)". Therefore, we omit it.
\end{proof}

For details in convex geometry, see \cite{MR3155183,MR2335496,MR3331351,MR4180684,MR3410930,MR3930585}; and for details in lattice theory, see \cite{MR3587859,MR3307662,MR2768581}.
	
\section{Pseudo-cones and the duality}\label{Sec:ProOfPC}
Let $C$ be a pointed closed convex cone with non-empty interior in $\R^n$.
We first show that $C$-close sets are all closed pseudo-cones. 
	
\begin{Lem}\label{Lem:OrignalNotInA}
	If 
	$K$ is a $C$-close convex set, then $o\notin K$.
\end{Lem}
\begin{proof}
	Since $C\setminus K$ is of positive volume, we have $C\setminus K\neq\emptyset$. By choosing $x\in C\setminus K$, $y\in\interior C$ such that $y\ne x$, together with the closedness of $K$, it gets
	\begin{equation*}\label{equ:Lem:OrignalNotInA:proof1}
	    \emptyset\neq(x,y]\setminus K\subset \interior C.
	\end{equation*}
	Let $B=B(z,\varepsilon)$ be a closed ball such that $z\in(x,y]\setminus K$ and $B\subset C\setminus K$.
		
	Suppose that $o\in K$. Together with the convexity of $K$, we have $([1,+\infty)B) \cap K = \emptyset$.
	Also, since $C$ is a convex cone,
	\begin{equation*}
    	[1,+\infty)B\subset C\setminus K.
	\end{equation*}
	Observe
	\begin{equation*}
	    (B\cap H_{z,\langle z,z\rangle}) + \R_+z\subset[1,+\infty) B,
	\end{equation*}
	the left hand side of which is a cylinder of infinite height. Thus, $C\setminus K$ is of infinite volume, a contradiction. It completes the proof.
\end{proof}
	
\begin{Thm}\label{Thm:CCloseIsPseducone}
If $K$ is a $C$-close convex set, then $K$ is a closed pseudo-cone.
\end{Thm}
\begin{proof}
	By Lemma \ref{Lem:OrignalNotInA}, $o\notin K$.	It is enough to show that
	\begin{equation*}
	    x\in K\text{ implies } \lambda x\in K \text{ for all } \lambda>1.
	\end{equation*}
	We verify it indirectly. Suppose $a\in K$ and $b=\lambda_0a\notin K$ with some $\lambda_0>1$. Since $K$ is a nonempty closed convex set, we may choose a hyperplane $H_{u,\alpha}\subset\R^n$ such that
	\begin{equation*}
	    K\subset \interior H^-_{u,\alpha}\text{ and }b\in \interior H^+_{u,\alpha}.
	\end{equation*}
	Thus, $\langle a,u\rangle<\alpha$ and $\langle \lambda_0a,u\rangle = \langle b,u\rangle>\alpha$. Together with $\lambda_0>1$, we get $\alpha>0$. It is then clear that $o \in\interior H^-_{u,\alpha}$.
	
	Notice $b \in C$. 
	If $b\in\bd C$, we can choose $d\in\interior C$ such that $d \notin H^-_{u,\alpha}$. 
	If $b\in\interior C$, let $d=b$. Write $K'=C\cap H^-_{u,\alpha}$.
	Choose $\varepsilon_0>0$ such that $B(d,\varepsilon_0)\subset C\setminus H^-_{u,\alpha}$. Then $C\setminus K'$ is of positive volume. Observe that
	\begin{equation*}
	    C\setminus K'\subset C\setminus K,
	\end{equation*}
	which implies that $C\setminus K'$ is of finite volume. So, $K'$ is $C$-close. By Lemma \ref{Lem:OrignalNotInA}, we have $o\notin K'$, which contradicts $o \in H^-_{u,\alpha}$. It completes the proof.
\end{proof}
	
	
The following lemma gives the recession cone of a closed pseudo-cone.

\begin{Lem}\label{Lem:ClR+K=RecK}
	If $K$ is a closed pseudo-cone, then $\cl\R_+K=\rec K$.
\end{Lem}
\begin{proof}
	First we prove $\cl\R_+K\subset\rec K$. Put $\lambda a_1\in\R_+K$ with $\lambda\geq0$ and $a_1\in K$, and put $a_2\in K$. We have
	\begin{equation*}
	    \lambda a_1+a_2=(1+\lambda)(\frac{\lambda}{1+\lambda}a_1+\frac{1}{1+\lambda}a_2)\in(1+\lambda)K\subset K,
	\end{equation*}
	which deduces that $\R_+ K\subset\rec K$ and then $\cl\R_+ K\subset\rec K$, since $\rec K$ is always closed.
		
Second, we prove $\cl\R_+K\supset\rec K$  indirectly. Suppose $x\in\rec K\backslash(\cl\R_+K)$. Choose a hyperplane $H_{u,0}$ that separates $x$ and $\cl\R_+K$, such that $x\in\interior H^+_{u,0}$ and $\cl\R_+K\subset H^-_{u,0}$. Fix any $a\in K$. Observe that
\begin{equation*}
	\langle u, a+\lambda x\rangle>0\text{ for large enough }\lambda>0,
\end{equation*}
which contradicts that $a+\lambda x \in K + \rec K \subset K$ for all $\lambda>0$. It completes the proof.
\end{proof}
	
Now, we give an equivalent definition of a closed pseudo-cone.
	
\begin{Thm}\label{Thm:EquiDefOfPC}
	If $K$ is a nonempty closed convex set not containing the origin, then $K$ is a pseudo-cone if and only if $K\subset\rec K$.
\end{Thm}
	
\begin{proof}
By Lemma \ref{Lem:ClR+K=RecK}, it is only necessary to show that $K\subset\rec K$ implies that $K$ is a pseudo-cone. Let $a\in K$, it is clear that
\begin{equation*}
	\lambda a=a+(\lambda - 1)a\in K + \R_+ K \subset K + \R_+\rec K = K + \rec K \subset K,
\end{equation*}
where $\lambda\geq1$. Thus, $K$ is a pseudo-cone.
\end{proof}
	
\begin{Rem}
	In fact, a nonempty closed convex set $K$ is a closed convex cone or a closed pseudo-cone if and only if $K\subset\rec K$.
\end{Rem}

\begin{Lem}\label{Lem:SupOfPC}
    For pseudo-cones $K,L$, if $o\notin \cl[K,L]$, then $\cl[K,L]$ is a closed pseudo-cone.
\end{Lem}
\begin{proof}
    We prove it by contradiction. Suppose
    $$x\in\cl[K,L],~\lambda_0 x\notin\cl[K,L]\text{ with }\lambda_0>1.$$
    Choose a hyperplane $H_{u,\alpha}$ such that $\lambda_0 x\in\interior H^-_{u,\alpha}$ and $\cl[K,L]\subset\interior H^+_{u,\alpha}$. It is clear that $\alpha<\langle u,x\rangle<\lambda_0^{-1}\alpha$,
    which implies $\alpha<0$.
    
    We claim that $K,L\subset H^+_{u,0}$. In fact, if $y\in K$ such that $\langle u,y\rangle<0$, then $\langle u,\lambda y\rangle<\alpha$ for large enough $\lambda>1$. Thus, $\lambda y\in K\cap H^-_{u,\alpha}=\emptyset$, a contradiction. Therefore, $K\subset H^+_{u,0}$. Similarly, $L\subset H^+_{u,0}$. 
    
    Now $x \in \cl[K,L]\subset H^+_{u,0}$, which contradicts $x\in\interior H^-_{u,\lambda_0^{-1}\alpha}$.
\end{proof}


The new duality of $\PC$ is investigated below. First, we show that the operator $\ast$ characterize closed pseudo-cones.
	
\begin{Lem}\label{Lem:A*IsPC}
	If $K$ is a nonempty closed convex set not containing the origin, then $K^\ast$ is a closed pseudo-cone.
\end{Lem}
\begin{proof}
	For every $y\in K$, it is easy to see that the half space $\{x\in\R^n:\langle x,y\rangle\leq -1\}$ is closed and convex. 
	Observe that
	\begin{equation*}
	    K^\ast=\bigcap_{y\in K}\{x\in\R^n:\langle x,y\rangle\leq -1\},
	\end{equation*}
	which implies $K^\ast$ is also closed and convex.
	Notice that $o \notin K$.
	There is a $u \in \mathbb{S}^{n-1}$ and $\alpha <0$ such that $K \subset H_{u,\alpha}^-$.
	It turns out $u/(-\alpha) \in K^\ast$ which implies $K^\ast \neq \emptyset$.
	Let $z\in K^\ast,w\in K$ and $\lambda>1$. Since $\langle \lambda z,w\rangle\leq-\lambda<-1$, we have $\lambda z\in K^\ast$. Together with the clear fact $o\notin K^\ast$, the proof is completed.
\end{proof}	

For a closed pseudo-cone $K$, the following lemma 
and Lemma \ref{Lem:ClR+K=RecK} give $(\rec K)^\circ=\rec(K^\ast)$.
	
\begin{Lem}\label{Lem:SwapCLR&*}
If $K\subset\R^n$ is a closed pseudo-cone, then $\cl(\R_+(K^\ast))=(\cl(\R_+K))^\circ$.
\end{Lem}
	
\begin{proof}

For $\lambda y\in\R_+(K^\ast)$ with $\lambda \ge 0$ and $y\in K^\ast$, it is clear $\langle y, K\rangle\leq -1$, which implies 
$$\langle\lambda y,\cl(\R_+ K)\rangle \leq 0.$$
Notice that $\cl(\R_+K)$ is a closed convex cone. 
It turns out $\cl(\R_+(K^\ast)) \subset (\cl(\R_+K))^\circ$ since $(\cl(\R_+K))^\circ$ is closed.

Let $x\in(\cl(\R_+K))^\circ$, then $\langle x,\cl(\R_+K)\rangle\leq0$, so $\langle x,K\rangle\leq0$. 
Write 
$$\eta=-\sup_{y \in K}\langle x,y\rangle,$$ 
which is nonnegative. 
If $\eta>0$, then $\sup_{y \in K}\langle\eta^{-1}x,y\rangle=-1$, which deduces that $\eta^{-1}x\in K^\ast$, so $x\in\R_+(K^\ast)$. 
If $\eta=0$, since $o\notin K$, there is a $u\in\mathbb{S}^{n-1}$ such that $\sup_{y \in K}\langle u,y\rangle=-\beta$ for some $\beta>0$.
We have
\begin{equation*}
\sup_{y \in K}\langle (1-\lambda)x+\lambda u,y\rangle\leq (1-\lambda)\sup_{y \in K}\langle x,y\rangle+\lambda\sup_{y \in K}\langle u,y\rangle=-\lambda\beta,\quad\lambda\in (0,1],
\end{equation*}
which implies that $\sup_{y \in K}\langle z,y\rangle<0$ for $z\in(x,u]$. 
By the conclusion of the case $\eta>0$, we get $(x,u]\subset\R_+(K^\ast)$, and then $x\in\cl(\R_+(K^\ast))$, which completes the proof.
\end{proof}

Next, we show that $\ast$ is an involution on the class of closed pseudo-cones , i.e. $K^{\ast\ast}=K$ for all closed pseudo-cones $K$.

\begin{Thm}\label{Thm:Duality}
	If $K$ is a closed pseudo-cone, then $K^{\ast\ast}=K$.
	Moreover, if $K$ is a closed convex set not containing the origin, then $K^{\ast\ast\ast}=K^\ast$.
\end{Thm}

\begin{proof}
If $x\in K$, for any $y\in K^\ast$, we have $\langle x,y\rangle\leq-1$, i.e. $x\in K^{\ast\ast}$.
Therefore, $K \subset K^{\ast\ast}$.

Let $K$ be a closed pseudo-cone.
Suppose that there is a $x\in K^{\ast\ast} \setminus K$. 
Observe that $\langle x,K^\ast\rangle\leq-1$. 
It turns out $\langle x,\cl(\R_+(K^\ast))\rangle\leq0$.
Together with Lemma \ref{Lem:SwapCLR&*},  
$$x \in (\cl(\R_+(K^\ast)))^\circ=\cl(\R_+K).$$
We write
\begin{equation*}
	u=\frac{x-p_K(x)}{|x-p_K(x)|}.
\end{equation*}
It is clear that $K\subset H_{u,\langle p_K(x),u\rangle}^-$ and $x\in\interior H^+_{u,\langle p_K(x),u\rangle}$.
		
If $\langle p_K(x),u\rangle>0$, then $o\in\interior H^-_{u,\langle p_K(x),u\rangle}$, which gives
\begin{equation*}
	(1,+\infty)p_K(x)\subset\interior H^+_{u,\langle p_K(x),u\rangle}.
\end{equation*}
It contradicts $(1,+\infty)p_K(x) \subset K$.
		
If $\langle p_K(x),u\rangle=0$, recalling that $x\in\cl(\R_+K)$,
then we can choose $z \in \relint(\R_+K)$ such that $z\in \interior H^+_{u,\langle p_K(x),u\rangle}$.
Therefore, $(0,+\infty)z \subset \interior H^+_{u,\langle p_K(x),u\rangle}$, a contradiction to $(0,+\infty)z \cap K \neq \emptyset$.
		
If $\langle p_K(x),u\rangle<0$, then $\langle K,u\rangle\leq\langle p_K(x),u\rangle<0$, which deduces that $u\in\R_+(K^\ast)\backslash\{o\}$. So, $0>\rho_{K^\ast}(u)>-\infty$.
Observe that $\langle x,-\rho_{K^\ast}(u)u\rangle \leq -1$ since $x \in K^{\ast\ast}$ and $-\rho_{K^\ast}(u)u \in K^\ast$.
Together with $K\subset H_{u,\langle p_K(x),u\rangle}^-$, we have
\begin{equation*}
    \begin{array}{r@{~}l}
    \langle K,-\rho_{K^\ast}(u)u\rangle 
    &\leq \langle p_K(x),-\rho_{K^\ast}(u)u\rangle \\
    &= -\langle x-p_K(x),-\rho_{K^\ast}(u)u\rangle+\langle x,-\rho_{K^\ast}(u)u\rangle\\
    &\le\rho_{K^\ast}(u)|x-p_K(x)| -1 \\
    &<-1,
\end{array}
\end{equation*}
which implies that $\langle K,-\lambda\rho_{K^\ast}(u)u\rangle\leq-1$ for some $\lambda\in(0,1)$. 
Thus, $-\lambda\rho_{K^\ast}(u)u\in K^\ast$.
It turns out $-\lambda\rho_{K^\ast}(u)\geq -\rho_{K^\ast}(u)$, i.e. $\lambda\geq 1$, a contradiction. It completes the proof of the first statement.
		
The second statement follows directly from the first statement and Lemma \ref{Lem:A*IsPC}.
\end{proof}
Next theorem shows the connections among the support function, the radial function, and $*$ of pseudo-cones.

\begin{Thm}
Let $K$ be a closed pseudo-cone and $u\in\R^n$. 
If $h_{K^*}(u)\geq0$, then $u\notin\R_+K\setminus\{o\}$.
If $h_{K^*}(u)<0$, then $u\in\R_+K\setminus\{o\}$ and $$\rho_K(u)=\frac{1}{h_{K^*}(u)}.$$ 
\end{Thm}
\begin{proof}
First, we consider the case $h_{K^*}(u)\geq 0$. 
If $u=o$, there is nothing to proof. 
Assume $u\neq o$. We prove it by contradiction. 
Suppose $u \in\R_+K\setminus\{o\}$. One has $\lambda_0u\in K$ for some $\lambda_0>0$.
Therefore, 
\begin{gather*}
\langle u,x\rangle\leq -\lambda_0^{-1}\text{ for all }x\in K^*.
\end{gather*}
It follows that $h_{K^*}(u)<0$, a contradiction.
    
Next, we consider the case $h_{K^*}(u)<0$. It is clear that $u\neq o$. One has
\begin{gather*}
    \langle u,x\rangle\leq h_{K^*}(u)\text{ for all }x\in K^*,
\end{gather*}
which implies that
\begin{gather*}
\left\langle \frac{u}{-h_{K^*}(u)},x \right\rangle\leq -1 \text{ for all }x\in K^*.
\end{gather*}
Together with Theorem \ref{Thm:Duality},  $\frac{u}{-h_{K^*}(u)}\in K^{**}=K$.  Hence, $u\in\R_+K\setminus\{o\}$.
It turns out that
\begin{gather*}
\begin{array}{r@{~}l}
\rho_K(u)&=-\min\{\lambda>0 : \lambda u\in K\}\\
&=-\min\{\lambda>0:\lambda u\in K^{**}\}\\
&=-\min\{\lambda>0:\langle\lambda u,x\rangle\leq -1\text{ for all }x\in K^*\}\\
&=-\min\{\lambda>0:\langle u,x\rangle\leq -\lambda^{-1}\text{ for all }x\in K^*\}\\
&=\left(\max\{\langle u,x\rangle:x\in K^*\}\right)^{-1}\\
&=\frac{1}{h_{K^*}(u)},\\
\end{array}
\end{gather*}
which completes the case $h_{K^*}(u)<0$.
\end{proof}

In the last part of this section, we study the lattice structure of $\PC$.

\begin{Lem}\label{Lem:PC^nIsLattice}
The partially ordered set $(\PC,\subset)$ is a lattice with
\begin{gather*}\label{eq:03-12-1056}
K\land L=K\cap L,
\end{gather*}
and
\begin{gather}\label{eq:313-1}
K\lor L =\begin{cases}
\cl[K, L], & o \notin \cl [K,L],\\
\R^n, & o\in \cl[K,L].
\end{cases}
\end{gather}
\end{Lem}
	
\begin{proof}
Let $K,L\in\PC$. Since $\emptyset\subset K,L\subset\R^n$, we get the fact that the pair $\{K,L\}$ has both upper bounds and lower bounds.

If $M$ is a lower bound of $\{K,L\}$, i.e. $M\subset K,L$, then $M\subset K\cap L$. Together with the fact that $K\cap L$ is also a lower bound of $\{K,L\}$, we get $K\land L=K\cap L$.

If $N$ is an upper bound of  $\{K,L\}$, then we have $K,L\subset N$, which implies $\cl [K,L]\subset N$. 
If $o\notin\cl[K,L]$, then Lemma \ref{Lem:SupOfPC} shows that $\cl[K,L]$ is an upper bound of $\{K,L\}$. 
If $o\in\cl[K,L]$, then $N \in \PC$ must be $\R^n$, since the only element of $\PC$ containing $o$ is $\R^n$ by definition.
Therefore, we get the desired result.
\end{proof}


\begin{Lem}\label{Lem:DualityOfSupInf}
If $K,L\in\PC$, then
\begin{gather*}
	K^\ast\land L^\ast = (K \lor L)^\ast,\\
	K^\ast\lor L^\ast = (K\land L)^\ast.
\end{gather*}
\end{Lem}

\begin{proof}
If $K\subset L$, then, by the definition of $\ast$, it is clear that $L^\ast\subset K^\ast$. Conversely, if  $L^\ast\subset K^\ast$, then, by the definition of $\ast$ and Theorem \ref{Thm:Duality}, it is clear that $K=(K^\ast)^\ast \subset  (L^\ast)^\ast=L$. 
The conclusion then follows from Lemma \ref{Lem:DualForLattice} and Theorem \ref{Thm:Duality} directly. 
\end{proof}

By the definition of $\ast$, the following lemma is trivial.
\begin{Lem}
If $K \in \PC$ and $g \in \mathrm{GL}(n)$, then $(gK)^\ast=g^{-t}K^\ast$.
\end{Lem}


	
\begin{proof}[\textbf{The equivalence of Theorem \ref{Thm:MainThm_Dual} and Theorem \ref{Thm:MainThm_Identity}.}]
Let $\varphi(K)=\tau(K^\ast)$ for any $K \in \PC$ (by Theorem \ref{Thm:Duality}, it is well-defined). 
By Lemma \ref{Lem:DualityOfSupInf},
that $\varphi$ is an endomorphism is equivalent to that $\tau$ interchanges $\vee$ and $\wedge$.
The proof is then completed by Theorem \ref{Thm:Duality}. 
\end{proof}
	
\section{Endomorphisms}\label{Sec:Endo}
In this section, we prove Theorem \ref{Thm:MainThm_Identity} including $n=1$.
In the rest of this paper, we always assume that $\varphi$ is a non-constant endomorphism of $\PC$.



	
We need the following lemma which is called the main theorem in affine geometry in \cite{MR3439685}. For details and proofs, see \cite{MR1138277,MR97769,MR1188083,MR1898158}.
	
\begin{Lem}\label{Lem:Affinity}
Let $n\geq 2$ and  $f\colon\R^n\rightarrow\R^n$ be an injection such that\\
\indent $(i)$ $f$ maps collinear points to collinear points,\\
\indent $(ii)$ $f(\R^n)$ is not contained in a line.\\
Then $f$ is a non-singular affinity.
\end{Lem}

\begin{Lem}\label{Lem:0706-2239}
	Let $m\ge 1$ be an integer. If $x_1,x_2,\ldots,x_m\in\R^n$ with $o\notin[x_1,x_2,\ldots,x_m]$, then
	\[\overline{x_1}\lor\overline{x_2}\lor\cdots\lor\overline{x_m}=[1,+\infty)[x_1,x_2,\ldots,x_m].\]
\end{Lem}
\begin{proof}
	Observe that $[1,+\infty)[x_1,x_2,\ldots,x_m]=[\overline{x_1},\overline{x_2},\ldots,\overline{x_m}]\subset \overline{x_1}\lor\overline{x_2}\lor\ldots\lor\overline{x_m}$. It is sufficient to verify that $[1,+\infty)[x_1,x_2,\ldots,x_m]$ is closed.
	
	Assume that  $\{\lambda_ka_k\}_{k=1}^\infty$ is a convergent sequence in $[1,+\infty)[x_1,x_2,\ldots,x_m]$ with $\lambda_k\ge1$ and $a_k\in[x_1,x_2,\ldots,x_m]$. Choose $r,R>0$ such that $B(o,r)\cap [x_1,x_2,\ldots,x_m]=\emptyset$ and $|\lambda_ka_k|\le R$ for all $k$. It is not hard to see that $\lambda_k\le r^{-1}R$ for all $k$. Hence, there is a strictly increasing sequence of integers $\{k_i\}_{i=1}^\infty$, $\lambda_0>0$ and $a_0\in\R^n$ such that $\lim_{i\to\infty}\lambda_{k_i}=\lambda_0$ and $\lim_{i\to\infty}a_{k_i}=a_0$. One gets $\lambda_0\in[1,+\infty)$ and $a_0\in[x_1,x_2,\ldots,x_m]$ by the fact that $[1,+\infty)$ and $[x_1,x_2,\ldots,x_m]$ are both closed. It follows that $\lim_{k\to\infty}\lambda_ka_k=\lim_{i\to\infty}\lambda_{k_i}a_{k_i}\in[1,+\infty)[x_1,x_2,\ldots,x_m]$. The proof is completed.
\end{proof}

\begin{Lem}\label{Lem:const}
$\varphi(\overline o) \neq \varphi(\emptyset)$.
\end{Lem}

\begin{proof}
Let $K \in \PC$.
Observe that $\emptyset\subset K\subset\R^n=\overline{o}$. By Lemma \ref{Lem:EndoIsInc}, we have 
\begin{gather*}
	\varphi(\emptyset)\subset\varphi(K)\subset\varphi(\overline{o}).
\end{gather*} 
If $\varphi(\overline o) = \varphi(\emptyset)$, then $\varphi$ is a constant, a contradiction.
\end{proof}

\subsection{Endomorphisms for \texorpdfstring{$n\geq2$}{n>=2}}
\begin{Lem}\label{Lem:one2all}
If $n\geq2$, then $\varphi(\overline x)\neq\varphi(\emptyset)$ for all $x\in\R^n$.
\end{Lem}

\begin{proof}
Set	$K=\{x\in\R^n:\varphi(\overline x)=\varphi(\emptyset)\}$. By Lemma \ref{Lem:const}, it is clear that $o\notin K$. Suppose that $K\neq\emptyset$. 

Next, we show that $K$ is a pseudo-cone. In fact, let $a,b\in K, c\in [a,b]$, and thus, $\emptyset\subset\overline c\subset\overline a\lor\overline b$. By Lemma \ref{Lem:EndoIsInc}, it deduces that
\begin{gather*}
	\varphi(\emptyset)\subset\varphi(\overline c)\subset\varphi(\overline a\lor\overline b)=\varphi(\overline a)\lor\varphi(\overline b)=\varphi(\emptyset),
\end{gather*}
i.e. $c\in K$.
Since $\emptyset\subset\overline{\lambda a}\subset\overline a$ for all $\lambda\geq 1$, by Lemma \ref{Lem:EndoIsInc}, we have
\begin{gather*}
	a\in K\Rightarrow\lambda a\in K \text{ for all } \lambda\geq1.
\end{gather*}

Since $K$ is a pseudo-cone, one may choose $u\in\mathbb{S}^{n-1}$ such that $K\subset H^-_{u,0}$, and then choose $y\in K$, $v\in H_{u,0}\setminus\{o\}$. 
Set $z=2y$, and let
\begin{gather*}
	w=\begin{cases}
		u+v,  & y=-|y|u, \\
		u,  & \text{ other cases},
	\end{cases}
\end{gather*}
$s=2w-y$ and $t=2w-2y$. It is trivial to show that $w=[y,s]\cap[z,t]$, $y,z\in K$, $w,s,t\in\interior H^+_{u,0}$. 
Clearly, $w$ and $y$ are linearly independent, so are $s$ and $t$. Observe that
\begin{align*}
\varphi(\overline w)&\subset\varphi\left((\overline y\lor\overline s)\land(\overline z\lor\overline t)\right)=\varphi(\overline y\lor\overline s)\land\varphi(\overline z\lor\overline t)=\left(\varphi(\overline y)\lor\varphi(\overline s)\right)\land\left(\varphi(\overline z)\lor\varphi(\overline t)\right)
\end{align*}
Since $y,z \in K$, we further have
\begin{align*}
\left(\varphi(\overline y)\lor\varphi(\overline s)\right)\land\left(\varphi(\overline z)\lor\varphi(\overline t)\right)
=\varphi(\overline s)\land\varphi(\overline t)=\varphi(\overline s\land\overline t)=\varphi(\emptyset).
\end{align*}
Together with $\varphi(\emptyset)\subset\varphi(\overline w)$, it implies $w\in K$, a contradiction. Hence, $K=\emptyset$.
\end{proof}

\begin{Lem}\label{Lem:PhiPSIsNotBarO}
    If $n\geq 2$, $K$ is a closed pseudo-cone, then $\varphi(K)\neq \R^n$.
\end{Lem}

\begin{proof}
One proves it by contradiction. Suppose $\varphi(K)= \R^n$. There exits a hyperplane  $H_{u,\alpha}$ separate $K$ and $o$ such that $K\subset\interior H_{u,\alpha}^-$. It is clear that
\begin{gather*}
    \varphi(\overline u)=\varphi(K)\land\varphi(\overline u)=\varphi(\emptyset)
\end{gather*}
which contradict Lemma \ref{Lem:one2all}. It completes the proof.
\end{proof}

The following lemma is the crucial topological proposition that is used in the classification.
	
\begin{Lem}\label{Lem:countable}
Let $n\geq 2$, $K \in \PC$ and $\mathcal{F}=\{K_\gamma: \dim K_\gamma =n,~ K\subsetneq K_\gamma,~\gamma\in\Gamma\}$ be a nonempty subclass of $\PC$. 
If $K_\alpha\cap K_\beta=K$ for any $\alpha,\beta \in \Gamma$ whenever $\alpha\neq\beta$, then $\mathcal F$ is at most countable. 
\end{Lem}

\begin{proof}
For each $\gamma\in\Gamma$, observe that 
$\interior (K_\gamma\setminus K) \neq \emptyset$ (since $K\subsetneq K_\gamma$).
We may associate a rational point $a_\gamma$ such that $a_\gamma\in K_\gamma\setminus K$. If $\alpha\neq\beta$, then $\interior(K_\alpha\setminus K)\cap\interior(K_\beta\setminus K)=\emptyset$. Therefore $a_\alpha\neq a_\beta$.
Hence, $\mathcal F$ is at most countable.
\end{proof}

\begin{Lem}\label{Lem:EmptysetMapstoEmptyset}
	If $n\geq2$, then $\varphi(\emptyset)=\emptyset$.
\end{Lem}
\begin{proof}
Suppose $\varphi(\emptyset)\neq\emptyset$. Then $m:=\dim \varphi(\emptyset)\geq 1$.

By Lemma \ref{Lem:one2all}, $\varphi(\overline y)\neq \varphi(\emptyset)$ for all $y\in\R^n$. 
Therefore, by Lemma \ref{Lem:EndoIsInc}, $\varphi(\emptyset) \neq \R^n$ which gives $o\notin \varphi(\emptyset)$.
\myvskip

\noindent\textit{Assertion I. } $\mathcal{U}:=\{u\in\mathbb{S}^{n-1}:\varphi(\overline{\lambda u})\subset\lin \varphi(\emptyset)\text{ for some }\lambda>0\}$ is at most countable.

\myvskip

In fact, for $u\in\mathcal{U}$, choose a fixed $\lambda_u$ such that $	    \varphi(\overline{\lambda_u u})\subset\lin \varphi(\emptyset)$.
Consider Lemma \ref{Lem:countable} in $\lin \varphi(\emptyset)$ with
\begin{gather*}
	\mathcal{F}:=\{\varphi(\overline{\lambda_uu}):u\in\mathcal{U}\},~K=\varphi(\emptyset).
\end{gather*}
We get that $\mathcal{F}$ is at most countable, i.e. $\mathcal{U}$ is at most countable. It completes the proof of Assertion I.

\myvskip
\noindent\textit{Assertion II. } $m \leq n-2$.
\myvskip

	Indeed, by Assertion I, it is clear that $m\neq n$, otherwise $\mathcal{U}=\mathbb{S}^{n-1}$. 
	Suppose $m= n-1$. 
	Let $\mathcal{V}=\mathbb{S}^{n-1}\setminus\mathcal{U}$. 
	Since $\varphi(\emptyset) \subset \varphi(\overline v)$ for all $v\in\mathcal{V}$, we have
	$\dim\lin\varphi(\overline v)=n$. By Lemma \ref{Lem:countable}, $\mathcal{V}$ is at most countable, which contradicts that $\mathcal{U}$ is at most countable. It completes the proof of Assertion II.

\myvskip
\noindent\textit{Assertion III.} 
Assume $n \ge 3$.
For every $k$-flat $S$ with $1\leq k\leq n-m-1$ and $o\notin S$, there are points $x_1,x_2,\ldots,x_{k+1}\in S$ such that 
	\begin{gather*}\label{equ:Lem:EmptysetMapstoEmptyset1}
	\dim\varphi(\overline x_1\lor\overline x_2\lor\ldots\lor\overline x_{k+1})\geq k+m+1.
	\end{gather*}
\myvskip

It is proved by induction with respect to $k$. For $k=1$, since $\{x/|x|:x\in S\}$ is an uncountable subset of $\mathbb{S}^{n-1}$, together with Assertion I,
there must be an $x_1 \in S$ such that
$\varphi(\overline{x_1})\not\subset \lin \varphi(\emptyset)$.
Recall that $\varphi(\emptyset) \subset \varphi(\overline{x_1})$. 
We have $\dim\varphi(\overline x_1)\geq m+1$. 
Suppose
\begin{gather*}
    \dim\varphi(\overline x_1\lor\overline y)\leq m+1\text{ for all } y\in S.
\end{gather*}
Then
\begin{gather*}
    \varphi(\overline y)\subset \lin \varphi(\overline x_1\lor\overline y) = \lin \varphi(\overline{x_1})\text{ for all }y\in S.
\end{gather*}
By Assertion I, the set $\{y\in S:\dim\varphi(\overline y)=m\}$ is at most countable. Thus there are  
uncountably many $y\in S$ such that $\dim \varphi(\overline y) = m + 1$, which contradicts Lemma \ref{Lem:countable} (in the space $\lin \varphi(\overline x_1)$, $\mathcal{F}=\{\varphi(\overline y):y\in S\}$, $K=\varphi(\emptyset)$ and $\varphi(\emptyset) \subsetneq \varphi(\overline y)$ for all $y \in S$) since the map $S\ni y\mapsto \varphi(\overline y)\in\PC$ is injective. (Recall Lemma \ref{Lem:one2all}.) 
Hence
$\varphi(\overline{x_1}\lor\overline{x_2})\geq m+2$ for some $x_2\in S$.

For $k\geq 2$, assume that Assertion III holds for all $(k-1)$-flats not containing the origin. 
Choose a $(k-1)$-flat $T\subset S$, by the hypothesis of induction, there exists $x_1,x_2,\ldots,x_{k}\in T$ such that 
\begin{gather*}
    \dim\varphi(\overline {x_1}\lor\overline{x_2}\lor\cdots\lor \overline{x_k})\geq k+m.
\end{gather*}
Set $K_T=\varphi(\overline {x_1}\lor\overline{x_2}\lor\cdots\lor \overline{x_{k}})$. 

Suppose
\begin{gather}\label{equ:0715-1728}
    \dim\varphi(\overline {x_1}\lor\overline{x_2}\lor\cdots\lor \overline{x_{k}}\lor\overline y)\leq k+m\text{ for all }y\in S.
\end{gather}
Let $\mathcal T$ be the family of all $(k-1)$-flats parallel to $T$ in $S$ (We specify $T\in\mathcal T$). For any $T'\in\mathcal T$ such that $T'\neq T$,
again by the hypothesis of induction, there are $y_1,y_2,\ldots,y_k\in T'$ such that
\begin{gather*}
    \dim\varphi(\overline {y_1}\lor\overline{y_2}\lor\cdots\lor \overline{y_k})\geq k+m.
\end{gather*}
Set $K_{T'}=\varphi(\overline {y_1}\lor\overline{y_2}\lor\cdots\lor \overline{y_k})$. Observe that, as we can take $y=x_k$ in \eqref{equ:0715-1728}, $\dim \lin K_T =k+m$, and
\begin{gather*}
    \varphi(\overline y) \subset \varphi(\overline {x_1}\lor\overline{x_2}\lor\cdots\lor \overline{x_k}\lor\overline y) \subset \lin K_T \text{ for all } y \in S. 
\end{gather*}
Hence, $\dim K_T=\dim K_{T'}=k+m$. 
Also, $K_{T'}\cap K_{T''}=\varphi(\emptyset)$ for distinguished $T',T''\in\mathcal T$.
Indeed, by setting $K_{T'}=\varphi(\overline{z'_1}\lor\overline{z'_2}\lor\ldots\overline{z'_{k}})$ and $K_{T''}=\varphi(\overline{z''_1}\lor\overline{z''_2}\lor\ldots\overline{z''_{k}})$ with $z'_1,z'_2,\ldots,z'_k\in T'$ and $z''_1,z''_2,\ldots,z''_k\in T''$, since Lemma \ref{Lem:0706-2239}, we have that
\begin{gather*}
(\overline{z'_1}\lor\overline{z'_2}\lor\ldots\overline{z'_{k}})\land(\overline{z''_1}\lor\overline{z''_2}\lor\ldots\overline{z''_{k}})=\big([1,+\infty)[z'_1,z'_2,\ldots,z'_k]\big)\cap\big([1,+\infty)[z''_1,z''_2,\ldots,z''_k]\big)
=\emptyset.
\end{gather*}
(Notice $o\notin S$.)
It follows that $K_{T'}\cap K_{T''}=K_{T'}\land K_{T''}=\varphi(\emptyset)$.
 
Together with the arbitrariness of $T'$, by Lemma \ref{Lem:countable} in $\lin K_T$, there are at most countably many such $T'$, a contradiction. It completes the induction.

\myvskip		

Now if $n=2$, then Assertion I contradicts Assertion II.
If $n \ge 3$, fix a $(n-1)$-flat $S$ which does not contain the origin.
We can find an uncountable family $\mathcal{T}$ of pairwise disjoint $(n-m-1)$-flats contained in $S$.
For any $T,T' \in \mathcal{T}$, by Assertion III for $k=n-m-1$,
we may associate closed pseudo-cone $K_T$ and $K_{T'}$ as in Assertion III such that $\dim K_T=\dim K_{T'}=n$, and that $K_T\cap K_{T'}=\varphi(\emptyset)$.
However, by Lemma \ref{Lem:countable}, $\mathcal{T}$ is countable, a contradiction.
\end{proof}

\begin{Rem}\label{Rem:0703-2225}
	For $K,L\in\PC$ with $K\land L=\emptyset$ and $L\ne\emptyset$, one has that $\varphi(K)\subsetneq \varphi(K\lor L)$, which is useful when we apply Lemma \ref{Lem:countable} later.
	In fact, by Lemma $\ref{Lem:EndoIsInc}$, $\varphi(K)\subset\varphi(K\lor L)$.
	Suppose $\varphi(K)=\varphi(K\lor L)$.
	It follows that $\varphi(L)\subset\varphi(K)\lor\varphi(L)=\varphi(K\lor L)=\varphi(K)$.
	One gets  $\varphi(L)=\varphi(K)\land\varphi(L)=\varphi(K\land L)=\varphi(\emptyset)=\emptyset$. Choose $x\in L$. Then, $\varphi(\overline x)\subset\varphi(L)=\emptyset$, i.e. $\varphi(\overline x)=\emptyset=\varphi(\emptyset)$, which contradicts Lemma \ref{Lem:one2all}. 
\end{Rem}

The following lemma shows that $\varphi$ maps a $1$-dimensional pseudo-cone to a $1$-dimensional pseudo-cone in the case $n\geq 3$.
\begin{Lem}\label{Lem:Point2PointForNGeq3}
	If $n\geq 3$, then $\dim{\varphi(\overline x)}=1$ for all $x\in\R^n\setminus\{o\}$.
\end{Lem}

\begin{proof}
	We process it indirectly. Suppose $\varphi(\overline p)\geq 2$ for some $p\in\R^n\setminus\{o\}$. 
	Set
	\begin{gather*}
	    F_u=\{\lambda p+\mu u:~\lambda\in\R,\mu>0\}\text{ for all }u\in \mathbb{S}^{n-1}\cap p^\bot.
	\end{gather*}
	Denote by $\mathcal{F}$ the family of all such $F_u$. It is clear that $\mathcal{F}$ is uncountable when $n\geq3$. By Lemma \ref{Lem:countable}, the family \begin{gather*}
	    \mathcal{F'}:=\{F\in\mathcal{F}:\varphi(\overline p\lor\overline y)\subset\lin\varphi(\overline p)\text{ for some }y\in F\}
	\end{gather*}
	is at most countable.
(By Remark \ref{Rem:0703-2225}, it is clear that $\varphi(\overline p)\subsetneq\varphi(\overline p\lor\overline y)$. 
Moreover, for distinguished $u,v\in\mathcal{S}^{n-1}\cap p^\bot$, choose $y_u\in F_u$ and $y_v\in F_v$ arbitrarily. Observe that $(\overline p\lor\overline{y_u})\land(\overline p\lor\overline{y_v})=\big([1,+\infty)[p,y_u]\big)\cap\big([1,+\infty)[p,y_v]\big)=\overline p$, since $p$, $y_u$ and $y_v$ are linearly independent.
Therefore, $\varphi(\overline p\lor\overline{y_u}) \cap \varphi(\overline p\lor\overline{y_v})=\varphi(\overline p)$, and then clearly $\varphi(\overline p\lor\overline{y_u}) \neq \varphi(\overline p\lor\overline{y_v})$.
Thus the countability of the family of such $\varphi(\overline p\lor\overline{y_u})$ implies the countability of $\mathcal F'$.)

\myvskip
\noindent\textit{Assertion I}. $\dim\varphi(\overline p)\leq n-1$.
\myvskip
    
	In fact, if $\lin\varphi(\overline p)=\R^n$, then $\mathcal{F}=\mathcal{F'}$, a contradiction. So, it completes Assertion I.

\myvskip		
\noindent\textit{Assertion II}. $\dim\varphi(\overline p)\leq n-2$.
\myvskip

	Indeed, suppose $\dim\varphi(\overline p)=n-1$.
	Similar to proving that $\mathcal  F'$ is at most countable, by replacing the surrounding space $\lin(\overline p)$ by $\R^n$, the family $\mathcal{F}\setminus\mathcal{F'}$ is at most countable, which is impossible. 
	Together with Assertion I, it completes Assertion II.
\myvskip		
\noindent\textit{Assertion III}. For every $k$-flat $S$ with $\dim\lin\{\overline p\cup S\}=k+2$ and $o\notin S$, there exists $x_1,\cdots,  x_{k+1}\in S$ such that
	\begin{gather*}
	    \dim\varphi(\overline p\lor\overline{x_1}\lor\overline{x_2}\lor\cdots\lor\overline{x_{k+1}})\geq\dim\varphi(\overline p)+k+1,
	\end{gather*}
	where $1\leq k\leq n-\dim\varphi(\overline p)-1$.
\myvskip
	The proof is by induction with respect to $k$, which is similar to that in Lemma \ref{Lem:EmptysetMapstoEmptyset}. For $k=1$, since $\mathcal{F'}$ is at most countable and
	$\mathcal{F''}:=\{F \in \mathcal{F}: F \cap S \neq \emptyset\}$ is uncountable, we may choose $x_1\in S$ such that $\dim\varphi(\overline p\lor\overline{x_1})\geq\dim\varphi(\overline p)+1$. Assume
	\begin{gather*}
	    \dim\varphi(\overline p\lor\overline{x_1}\lor\overline{y})\leq\dim\varphi(\overline p)+1 \text{ for all } y\in S,
	\end{gather*}
	then,
	\begin{gather*}
	    \varphi(\overline p\lor\overline{y})\subset\lin\varphi(\overline p\lor\overline{x_1})\text{ for all }y\in S.
	\end{gather*}
	By Lemma \ref{Lem:countable} with $K=\varphi(\overline p)$, the sets
	\begin{gather*}
	    \{y\in S:\dim\varphi(\overline p\lor\overline y)=\dim\varphi(\overline p)+1\}\text{ and }\{y\in S:\dim\varphi(\overline p\lor\overline y)=\dim\varphi(\overline p)\}
	\end{gather*}
	are both at most countable, a contradiction. (The proof is similar to proving that $\mathcal F'$ is at most countable.)
	For $k\geq 2$, by the hypothesis of induction, we may choose $(k-1)$-flat $T$ such that $T \subset S$, satisfying
	\begin{gather*}
	    \dim\varphi(\overline p\lor\overline{x_1}\lor\overline{x_2}\lor\cdots\lor\overline{x_{k}})\geq\dim\varphi(\overline p)+k
	\end{gather*}
	for some $x_1,x_2,\cdots,x_k\in T$. If
	\begin{gather*}
	    \dim\varphi(\overline p\lor\overline{x_1}\lor\overline{x_2}\lor\cdots\lor\overline{x_{k}}\lor\overline y)\leq\dim\varphi(\overline p)+k
	\end{gather*}
	for all $y\in S$, then
	\begin{gather*}
	    \varphi(\overline y)\subset\lin\varphi(\overline p\lor\overline{x_1}\lor\overline{x_2}\lor\cdots\lor\overline{x_{k}})\text{ for all }y\in S.
	\end{gather*}
	Set $K_T=\varphi(\overline p\lor\overline{x_1}\lor\overline{x_2}\lor\cdots\lor\overline{x_{k}})$.
	For every $(k-1)$-flat $T' \subset S$ parallel to $T$, it is valid that
	\begin{gather*}
	    \dim\varphi(\overline p\lor\overline{y_1}\lor\overline{y_2}\lor\cdots\lor\overline{y_{k}})\geq\dim\varphi(\overline p)+k
	\end{gather*}
	for some $y_1,y_2,\cdots,y_k\in T'$. 
	Set $K_{T'}=\varphi(\overline p\lor\overline{y_1}\lor\overline{y_2}\lor\cdots\lor\overline{y_{k}})$. 
	Clearly, $K_{T'}\subset\lin K_T$. 
    By Remark \ref{Rem:0703-2225}, $\varphi(\overline p) \subsetneq K_{T'}$.
    Also, we claim $K_{T'}\cap K_{T''}=\varphi(\overline p)$, if $T', T''\subset S$ are parallel to $T$ with $T'\cap T''=\emptyset$. 
	Thus Lemma \ref{Lem:countable} implies that there are at most countably many $(k-1)$-flats in $S$ disjoint with $T$, a contradiction. 
	It completes the proof of Assertion III.
	
	The proof of the claim is the following. Indeed, it is sufficient to show that
	$(\overline p\lor\overline{z'_1}\lor\overline{z'_2}\lor\ldots\overline{z'_{k}})\land(\overline p\lor\overline{z''_1}\lor\overline{z''_2}\lor\ldots\overline{z''_{k}})\subset\overline p$, where $z'_1,\ldots,z'_k\in T'$ and $z''_1,\ldots,z''_k\in T''$.
    Since $\dim\lin(\overline p\cup S)=k+2$ and that $T',T''\subset S$ are parallel $(k-1)$-flats, we may choose $p_1,p_2\in\R^n$, $(k-1)$ dimensional linear subspace $V$, and $(n-k-2)$ dimensional linear subspace $W$ such that $\R^n=W\oplus V\oplus(\R p)\oplus(\R p_1)\oplus(\R p_2)$ with $S=p_1+\R p_2+V$, $T'=p_1+\lambda'p_2+V$, $T''=p_1+\lambda''p_2+V$ for some distinguished $\lambda',\lambda''\in\R$.
    Here $\oplus$ is the direct sum of linear spaces.
    If $y\in(\overline p\lor\overline{z'_1}\lor\overline{z'_2}\lor\ldots\overline{z'_{k}})\land(\overline p\lor\overline{z''_1}\lor\overline{z''_2}\lor\ldots\overline{z''_{k}})$,
    by Lemma \ref{Lem:0706-2239},
    we have
    \[y=\mu'((1-\nu')p+\nu'(p_1+\lambda'p_2+v'))=\mu''((1-\nu'')p+\nu''(p_1+\lambda''p_2+v'')),\]
    where $\mu',\mu''\ge1$, $\nu',\nu''\in[0,1]$, $v',v''\in V$.
    Thus, $\mu'\nu'=\mu''\nu''$, $\mu'=\mu''$, $\nu'=\nu''$. Since $\lambda'\ne\lambda''$, $\nu'=\nu''=0$ holds. Hence, $y\in\overline p$.

\myvskip
Finally, fix an $(n-2)$-flat $S$ such that $\dim\lin(\overline p\cup S)=n$.
We can choose an uncountable family $\mathcal{T}$ of pairwise disjoint $(n-\dim \varphi(\overline p)-1)$-flats contained in $S$.
For any $T,T' \in \mathcal{T}$, by Assertion III for $k=n-\dim\varphi(\overline p)-1$,
we may associate closed pseudo-cone $K_T$ and $K_{T'}$ such that $\dim K_T=\dim K_{T'}=n$, $\varphi(\overline p) \subsetneq K_T$, $\varphi(\overline p) \subsetneq K_{T'}$ and $K_T\cap K_{T'}=\varphi(\overline p)$.
Thus Lemma \ref{Lem:countable} implies that $\mathcal{T}$ is countable, a contradiction.
\end{proof}

For $L$ is an arc in $\mathbb{S}^1$, we say $\relint L$ is the relative interior of $L$ with respect to $\mathbb{S}^1$.

By Lemma \ref{Lem:PhiPSIsNotBarO}, one has $o\notin\varphi(\overline{x})$ for all $x\neq o$ if $n \ge 2$. 
Write
$$\Phi\colon\R^n\setminus\{o\}\rightarrow\mathbb{S}^{n-1},~ x\mapsto |x|^{-1}x,$$
and for $n=2$,
\begin{gather*}
\sigma(x)=\Phi\left( \bigcup_{\lambda>0}\varphi(\overline{\lambda x})\right)\text{ for } x\in\R^2\setminus\{o\},\\
U'=\{x\in\mathbb{S}^1:\sigma(x)\text{ is a singleton}\}.
\end{gather*}

Notice that $\varphi(\overline{\lambda x})$ is monotone with respect to $\lambda$.
Thus $\bigcup_{\lambda>0}\varphi(\overline{\lambda x})$ is convex.
It turns out that $\sigma(x)$ contains an open arc (with respect to $\mathbb{S}^1$) if it is not a singleton.
In fact, suppose different $a,b\in\sigma(x)$. There exists $\mu_a,\mu_b,\lambda_0>0$ such that $\mu_aa,\mu_bb\in\varphi(\overline{\lambda_0x})$. If $a+b=o$, then $o\in\varphi(\overline{\lambda_0x})$, which implies that $\varphi(\overline{\lambda_0x})=\R^2$, $\varphi(x)=\mathbb{S}^1$. If $a+b\neq o$, by letting $L$ be an arc contained in a semicircle with $a$ and $b$ as endpoints, one gets that $\Phi([\mu_aa,\mu_bb])=L$. Thus, $L\subset\sigma(x)$. Hence, $\sigma(x)$ contains an open arc.

Observe that $\sigma(x)\cap\sigma(y)=\emptyset$ if $R_x\neq R_y$. To obtain it, a little manipulation is needed. Suppose $a\in\sigma(x)\cap\sigma(y)$, i.e.  $\mu_1a\in\varphi(\overline{\lambda_1x})\cap\varphi(\overline{\lambda_1y})$ for large enough $\mu_1>0$ and small enough $\lambda_1>0$.
Hence, $\mu_1a\in\varphi(\overline{\lambda_1x})\land\varphi(\overline{\lambda_1y})=\varphi(\overline{\lambda_1x}\land\overline{\lambda_1y})=\varphi(\emptyset)=\emptyset$ by Lemma \ref{Lem:EmptysetMapstoEmptyset}, a contradiction.

Therefore, $\mathbb{S}^1\setminus U'$ is a set of at most countably many points.

Let
\begin{gather}\label{eq:03-11-1300}
	U=\begin{cases}
		U'\cap(-U') 		& n=2;\\
		\mathbb{S}^{n-1}	& n\geq3.
	\end{cases}
\end{gather}

If $n=2$, then $\mathbb{S}^1\setminus U$ is a set of at most countably many points. 
Hence, for $n\geq2$, $U$ is dense in $\mathbb{S}^{n-1}$, and thus, $\R_{+}U$ is dense in $\R^n$.
We arrive at the following lemma which contains Lemma \ref{Lem:Point2PointForNGeq3}.
\begin{Lem}\label{Lem:Point2PointForNGeq2}
	If $n\geq2$, then $\dim{\varphi(\overline x)}=1$ for all $x\in \R_+U\setminus\{o\}$.
\end{Lem}
	
By Lemma \ref{Lem:Point2PointForNGeq2}, we set $f:\R_+U\setminus\{o\}\rightarrow\R^2 \setminus \{o\}$ for $n\geq2$ such that 
\begin{gather*}\label{eq:0308-2}
\varphi(\overline x)=\overline{f(x)}.   
\end{gather*} 
We define a partial (not strict) order $\preceq$ on $\R^n,n\geq1$ by
\begin{gather*}
    a\preceq b\text{ if and only if }a=\lambda b\text{ for some }\lambda\in[0,1].
\end{gather*}
Also, $a\prec b$ if and only if $a\preceq b$ and $a\neq b$.

\begin{Lem}\label{Lem:fIncreasing}
	If $n\geq2$, $a,b\in \R_+U \setminus \{o\}$ and $a \preceq b$, then $f(a)\preceq f(b)$.
\end{Lem}

\begin{proof}
It follows directly from Lemma \ref{Lem:EndoIsInc}.
\end{proof}

\begin{Lem}\label{Lem:OneRayMapstoOneRay}
	If $n\geq2$, $a,b\in \R_+U \setminus  \{o\}$ and $R_a\cap R_b=\emptyset$, then
	\begin{gather*}
		R_{f(a)}\cap R_{f(b)}=\emptyset.
	\end{gather*}
\end{Lem}

\begin{proof}
	If $R_a\cap R_b=\emptyset$, then $\overline a\land\overline b=\emptyset$, which deduces that $\varphi(\overline a)\land\varphi(\overline b)=\varphi(\emptyset)$.
	By Lemma \ref{Lem:EmptysetMapstoEmptyset}, $\varphi(\emptyset)=\emptyset$.	
	Therefore, $\overline{f(a)} \cap \overline {f(b)} =\emptyset$ which gives $R_{f(a)}\cap R_{f(b)}=\emptyset$.
\end{proof}

\begin{Lem}\label{Lem:AllMapstoAll}
	If $n\geq2$, then $\varphi(\overline o)=\overline o$. In particular, $R_{f(-a)}=R_{-f(a)}$ for all $a\in \R_+U\setminus\{o\}$.
\end{Lem}

\begin{proof}
	Suppose that $\varphi(\overline o)\neq\overline o$.
	Choose linearly independent $b, c\in \R_+U\setminus\{o\}$, then
	\begin{gather*}
	\varphi(\overline b)\lor\varphi(\overline{-b}) 
	= \varphi(\overline{o})
	=\varphi(\overline c)\lor\varphi(\overline{-c}).
	\end{gather*}
	Hence by \eqref{eq:313-1} and $o \notin \varphi(\overline o)$,
	\begin{align*}
	\varphi(\overline b)\lor\varphi(\overline{-b})&= [\overline{f(-b)}, \overline{f(b)}]\\
	\varphi(\overline c)\lor\varphi(\overline{-c})&=[\overline{f(-c)}, \overline{f(c)}].
	\end{align*}
	Thus $ [R_{f(-b)},R_{f(b)}]=\left(\R_{+} \varphi(\overline o)\right)\setminus\{o\} =[R_{f(-c)},R_{f(c)}]$. It follows that $R_{f(b)}=R_{f(c)}$ or $R_{f(b)}=R_{f(-c)}$.
	It contradicts $(R_{f(b)}\cup R_{f(-b)}) \cap (R_{f(c)}\cup R_{f(-c)})=\emptyset$ which follows from Lemma \ref{Lem:OneRayMapstoOneRay}.
	Therefore, $\varphi(\overline o)=\overline o$.
		
	Now for $a\in \R_+U\setminus\{o\}$, it is clear that
	\begin{gather*}
	    \overline{f(a)}\lor\overline{f(-a)}=\varphi(\overline a)\lor\varphi(\overline{-a})=\varphi(\overline o)=\overline o.
	\end{gather*}
	But both $f(a)$ and $f(-a)$ are not the origin. 
	Therefore, $f(-a)\in-R_{f(a)}$, i.e. $R_{f(-a)}=R_{-f(a)}$.
		
	It completes the proof.
\end{proof}

\begin{Rem}\label{Rem:Describef}
	Now, we could extend $f$ to a function mapping $\R_+U$ to $\R^n$, by defining $f(o)=o$ additionally. Hence, $a\preceq b \Rightarrow f(a)\preceq f(b)$, which generalizes Lemma \ref{Lem:fIncreasing}.
\end{Rem}	
    
\begin{Lem}\label{lem:0305-1}
	If $n\geq2$, $a,b,c \in \R_+U\setminus\{o\}$, then $c \in [R_a,R_b]$ implies
	\begin{gather*}
		f(c)\in[R_{f(a)},R_{f(b)}].  
	\end{gather*}
\end{Lem}

\begin{proof}
	In the case that $a$ and $b$ are linearly independent, we may assume that $c=\alpha a+\beta b$ with $\alpha,\beta>0$. Then $c\in \overline{\alpha a} \lor \overline{\beta b}$, which implies
	\begin{gather*}
	\overline{f(c)}=\varphi(\overline c)\subset \varphi(\overline{\alpha a}\lor\overline{\beta b})=\overline{f(\alpha a)}\lor\overline{f(\beta b)}
	\subset[R_{f(a)},R_{f(b)}].
	\end{gather*}
	Hence, $f(c)\in[R_{f(a)},R_{f(b)}]$.
	
	In the case $a\in R_b$, we have $R_a=R_b$. 
	Then $c\in R_a$. By Lemma \ref{Lem:fIncreasing}, $f(c)\in R_{f(a)}=[R_{f(a)},R_{f(b)}]$.

	In the case $-a\in R_b$, we have $c\in R_a$ or $c \in R_{b}$. By Lemma \ref{Lem:fIncreasing}, 
	$f(c)\in R_{f(a)}\cup R_{f(b)}\subset[R_{f(a)},R_{f(b)}]$.
\end{proof}

\begin{Lem}\label{Lem:Plane2Plane}
	If $n\geq 2$, $F$ is a 2-dimensional linear subspace of $\R^n$, then $\dim \lin f(F\cap \R_+U)=2$.
\end{Lem}
	
\begin{proof}
	Choose $u,v\in F\cap \R_+U$ such that $u$ and $v$ are linearly independent. By Lemma \ref{Lem:OneRayMapstoOneRay} and Lemma \ref{Lem:AllMapstoAll}, $R_{f(u)}$ and $R_{f(v)}$ are linearly independent. Clearly, it is enough to show that
	\begin{gather*}
	    x\in[R_u,R_v]\Rightarrow f(x)\in[R_{f(u)},R_{f(v)}],
	\end{gather*}
	which follows directly from Lemma \ref{lem:0305-1}.
\end{proof}
    
\begin{Lem}\label{Lem:FMapsRay2Ray}
	If $n\geq2$, then
	\begin{gather*}
		\lim_{\lambda\rightarrow+\infty}|f(\lambda x)|=+\infty\text{ for all }x\in \R_+U\setminus\{o\}.
	\end{gather*}
\end{Lem}
	
\begin{proof}
For any $a \in \R_+U\setminus\{o\}$, set $g(a)=\lim_{\lambda\rightarrow+\infty}|f(\lambda a)|$ . 
		
Choose an arbitrary $2$-dimensional subspace $F\subset\R^n$. By Lemma \ref{Lem:Plane2Plane}, $f(F \cap \R_+U)$ is a subset of a $2$-dimensional plane.
Let $W=U\cap F$. One has
	\begin{gather}\label{eq:03-11-1033}
		W \text{ is dense in } \mathbb{S}^{n-1}\cap F.
	\end{gather}
	We need to show that
	\begin{gather}\label{eq:03-11-1038}
	    g(a)=+\infty\text{ for all }a\in W.
	\end{gather}
Set 
$$h(a)=\inf_{\varepsilon>0}\sup_{b\in W\cap B(a,\varepsilon)} g(b) \text{ for any } a \in \mathbb{S}^{n-1}\cap F.$$
It is clear that $g(a)=+\infty$ implies $h(a)=+\infty$, where $a\in W$.

\myvskip
\noindent\textit{Assertion I.} Let $a,b\in \mathbb{S}^{n-1}\cap F$ be linear independent. If $h(a)=h(b)=+\infty$ and $c\in (\relint [R_a,R_b])\cap W$, then
	\begin{gather*}
	    g(c)=+\infty.
	\end{gather*}
\myvskip	
Indeed, suppose $g(c)<+\infty$. 
Choose $a',b' \in W$ with $a'+b' \neq 0$ such that $[R_a,R_b]\subset\relint[R_{a'},R_{b'}]$.
By Lemma \ref{lem:0305-1}, $f(c)\in [R_{f(a')},R_{f(b')}]$.
Since 
\begin{gather*}
  \bigcup_{\lambda>0} \left[\lambda \Phi(f(a')),\lambda\Phi(f(b'))\right]=[R_{f(a')},R_{f(b')}],
\end{gather*}
there is an $\alpha>0$ such that
\begin{gather*}
  g(c)\Phi(f(c))\in \left[\alpha \Phi(f(a')),\alpha\Phi(f(b'))\right].
\end{gather*}
The assumption $h(a)=h(b)=+\infty$ implies that there exist $d\in [R_{a'},R_{b'}] \cap B(a,\frac{|a-c|}{2}) \cap W$ and $e \in [R_{a'},R_{b'}] \cap B(b,\frac{|b-c|}{2}) \cap W$ and $\beta,\gamma>0$ such that $|f(\beta d)|>\alpha+1$, $|f(\gamma e)|>\alpha+1$.
By Lemma \ref{lem:0305-1}, $f(\beta d), f(\gamma e)\in [R_{f(a')},R_{f(b')}]$.
Hence,
\begin{gather}\label{eq:03-12-2042}
	g(c)\Phi(f(c)) \notin \overline{f(\beta d)}\lor\overline{f(\gamma e)}.
\end{gather}
Choose $c'\in R_c$ such that $c'\in\overline{\beta d}\lor\overline{\gamma e}$, and thus,
\begin{gather*}
f(c')\in\varphi(\overline {c'}) \subset\varphi(\overline{\beta d} \lor \overline{\gamma e}) =\overline{f(\beta d)}\lor\overline{f(\gamma e)}.
\end{gather*}
Together with \eqref{eq:03-12-2042}, it implies that
\begin{gather*}
	g(c)\Phi(f(c))\prec f(c').
\end{gather*}
One gets $|f(c')|>g(c)$, which contradicts the definition of $g$ (by Lemma \ref{Lem:fIncreasing}).  
It completes Assertion I.
\myvskip
Let $A=\{a\in \mathbb{S}^{n-1}\cap F:h(a)=+\infty\}$. 
Notice that Assertion I implies: Let $a,b\in\mathbb S^{n-1}\cap F$ be non-antipodal and such that $h(a) = h(b) = +\infty$. Then for every $c$ on the shortest arc connecting $a$ and $b$ we have $h(c) = +\infty$. Hence, this arc is contained in $A$.
If $A$ is not contained in any closed hemisphere, then $A =\mathbb S^{n-1}\cap F$.

If $A=\mathbb{S}^{n-1}\cap F$, by Assertion I, the desired statement is trivial. Suppose $A\neq \mathbb{S}^{n-1}\cap F$. By Assertion I, there is a closed
arc $L \subset \mathbb{S}^{n-1}\cap F$ which is contained in an open hemisphere such that
\begin{gather}\label{eq:03-10-2217}
	h(a)<+\infty\text{ for all }a\in L.
\end{gather}

For every $a\in L$, by $h(a)<+\infty$, we may choose $\varepsilon_a, \Xi_a>0$ such that $g(b)<\Xi_a$ for all $b\in W\cap B(a,\varepsilon_a)$. Since $L$ is compact, a finite family $\{B(a_i,\varepsilon_{a_i}):a_i\in L\}_{i=1}^k$ covers $L$. Put $\Xi=\max\{\Xi_{a_i},i=1,2,\cdots,k\}$.
It is clear that
\begin{gather}\label{eq:0308-1}
  g(a)<\Xi\text{ for all }a\in L\cap W.
\end{gather}
\myvskip
\noindent\textit{Assertion II.} For $a\in(\relint L)\cap W$, there exist an arc $L_a\subset\mathbb{S}^{n-1}\cap \lin f(\R_+ W)
$ such that $\Phi(f(a))$ is the midpoint of $L_a$ and that
\begin{gather*}
  \Phi(f(b))\notin L_a\text{ for all }b\in L\cap W\setminus\{a\}.
\end{gather*}

\myvskip
		
Indeed, let $c\in a^\bot\cap F$ such that $|c|=1$. For $b\in (H_{c,0}^+\cap L\cap W)\setminus\{a\}$, there exists a sufficiently large $\lambda_b>0$ such that $\overline{\lambda_b b}\subset \{\lambda a+\mu c\in\R^n:\lambda\in\R,\mu\in[1,+\infty)\}$.
Together with 
\eqref{eq:0308-1} and the ``increasing" of $f$ (Lemma \ref{Lem:fIncreasing}), we have
\begin{gather}\label{eq:0310-1}
  \Xi\Phi(f(b))\in \overline {f(\lambda_b b)} = \varphi(\overline{\lambda_b b})\subset\varphi(\{\lambda a+\mu c\in\R^n:\lambda\in\R,\mu\in[1,+\infty)\}).
\end{gather}
By Lemma \ref{Lem:EmptysetMapstoEmptyset}, together with $\overline a\cap \{\lambda a+\mu c\in\R^n:\lambda\in\R,\mu\in[1,+\infty)\}=\emptyset$, it follows that
\begin{gather}\label{eq:03-10-1309}
  \varphi(\overline a)\cap\varphi(\{\lambda a+\mu c\in\R^n:\lambda\in\R,\mu\in[1,+\infty)\})=\varphi(\emptyset)=\emptyset.
\end{gather}
Since \eqref{eq:0308-1}, we have $f(a)\preceq\Xi\Phi(f(a))$, i.e. $\Xi\Phi(f(a))\in\overline{f(a)}=\varphi(\overline a)$. By the closedness of $\varphi(\{\lambda a+\mu c\in\R^n:\lambda\in\R,\mu\in[1,+\infty)\})$ and \eqref{eq:03-10-1309}, there is some $\varepsilon_1>0$ such that
\begin{gather*}
  B\big(\Xi\Phi(f(a)),\Xi\varepsilon_1 \big)\cap \varphi(\{\lambda a+\mu c\in\R^n:\lambda\in\R,\mu\in[1,+\infty)\})=\emptyset.
\end{gather*}
Remembering \eqref{eq:0310-1}, it is clear that
\begin{gather*}
  \Xi\Phi(f(b))\notin B\big(\Xi\Phi(f(a)),\Xi\varepsilon_1 \big),
\end{gather*}
which implies
\begin{gather*}
  |\Phi(f(b))-\Phi(f(a))|=\Xi^{-1}|\Xi\Phi(f(b))-\Xi\Phi(f(a))|>\varepsilon_1.
\end{gather*}
Now, we get
\begin{gather}\label{eq:03-10-13191}
  B\big(\Phi(f(a)),\varepsilon_1 \big)
  \cap\bigcup_{b\in (H_{c,0}^+\cap  L)\setminus\{a\})} \Phi(f(b))=\emptyset.
\end{gather}
Similarly,
\begin{gather}\label{eq:03-10-13192}
  B(\Phi(f(a)),\varepsilon_2)\cap\bigcup_{b\in (H_{c,0}^-\cap  L)\setminus\{a\})} \Phi(f(b))=\emptyset
\end{gather}
for some $\varepsilon_2>0$. Let $\varepsilon=\min\{\varepsilon_1,\varepsilon_2\}$. Together with \eqref{eq:03-10-13191} and \eqref{eq:03-10-13192}, we have
\begin{gather*}
	B(\Phi(f(a)),\varepsilon)\cap\bigcup_{b\in L\setminus\{a\}} \Phi(f(b))=\emptyset.
\end{gather*}
Set
\begin{gather*}
	L_a:=B(\Phi(f(a)),\varepsilon)\cap\mathbb{S}^{n-1}\cap\lin f(\R_+W),
\end{gather*}
which completes the proof of Assertion II.
\myvskip
\noindent\textit{Assertion III. } The points on $(\relint L)\cap W$ are at most countably many.
\myvskip
In fact, choose $L'_a$ be an open arc with the same midpoint and a half length of $L_a$ in Assertion II. Hence, every point $a$ on $(\relint L)\cap W$ has image $\Phi(f(a))$ belonging to the arc $L'_a$, and all these arcs are disjoint pairwise. 
One gets that the points on $(\relint L)\cap W$ are at most countable many. It completes the proof of Assertion III.

\myvskip

Finally, by \eqref{eq:03-11-1033}, Assertion III is not possible, thus, \eqref{eq:03-11-1038} is verified, which completes the proof.
\end{proof}

\begin{Lem}\label{Lem:AlwaysRegular}
  If $n\geq2$, $a\in \R_+U\setminus\{o\}$, then there is some $b$ such that 
  $a\prec b$ and $f(a)\prec f(b)$.
\end{Lem}
\begin{proof}
  By Lemma \ref{Lem:FMapsRay2Ray}, $\lim_{\lambda\rightarrow+\infty}|f(\lambda a)|=+\infty$.
  Together with Lemma \ref{Lem:fIncreasing},
  $f(\lambda_0 a)\in R_{f(a)}$ and $|f(\lambda_0 a)|>|f(a)|$ for sufficiently large $\lambda_0>1$, which implies $f(a)\prec f(\lambda_0 a)$. 
  Put $b=\lambda_0 a$. The proof is completed.
\end{proof}

For $o \neq b \in \R_+U$, if there is $a\neq o$ such that $a\prec b$ and $f(a)\prec f(b)$, then $b$ is called a \textit{regular point}. Clearly, if $b$ is a regular point, and $b\prec c$, then $c$ is also a regular point.
	
\begin{Lem}\label{Lem:Collinear}
	If $n\geq2$, then
	\begin{gather*}
		a,b,c\in \R_+U\text{ are collinear in this order }\Rightarrow f(a),f(b),f(c) \text{ are collinear in this order.}
	\end{gather*}
\end{Lem}
	
\begin{proof}
Without loss of generality, we assume $a,b,c$ are distinguishable. 
If $o\in\aff\{a,b,c\}$, by Lemma \ref{Lem:fIncreasing} and Lemma \ref{Lem:AllMapstoAll}, there is nothing to prove. 
Now, suppose that $o\notin\aff\{a,b,c\}$.
\myvskip
\noindent\textit{Assertion I.} If $b$ is a regular point, then $f(a),f(b),f(c)$ are collinear in this order.
	
\myvskip
	
In fact, since $\overline b\subset\overline a\lor\overline c$, we have
\begin{gather*}
	f(b)\in\overline{f(b)}\subset\overline{f(a)}\lor\overline{f(c)}.
\end{gather*}
By the regularity of $b$, there is $b'\prec b$ such that $f(b')\prec f(b)$.
If $f(b)\notin[f(a),f(c)]$, then
\begin{gather*}
	\overline{f(b)}\cap(\overline{f(a)}\lor\overline{f(c)})\subsetneq\overline{f(b')}\cap(\overline{f(a)}\lor\overline{f(c)}),
\end{gather*}
i.e. $\varphi(\overline b\land(\overline a\lor\overline c))\subsetneq\varphi(\overline{b'}\land(\overline a\lor\overline c))$, which contradicts $\overline b\land(\overline a\lor\overline c)=\overline{b'}\land(\overline a\lor\overline c)$. 
It completes the proof of Assertion I.
\myvskip
\noindent\textit{Assertion II.} Every point $x\in \R_+U\setminus\{o\}$ is regular.
\myvskip
Indeed, by Lemma \ref{Lem:AlwaysRegular}, we may choose a regular $y\in \R_+U$ linearly independent from $x$.  
Further choose $z\in \R_+U$ such that $y\in\relint[x,z]$. 
By Lemma \ref{Lem:AlwaysRegular}, there is $z\prec z'$ such that $f(z)\prec f(z')$. By the convexity of the triangle $[o,x,z]$, the line $\aff\{y,z'\}$ intersects $\relint[o,x]$ at a point $x'\in \R_+U$.
		
By Assertion I, $f(x),f(y),f(z)$ are collinear, as well as $f(x'),f(y),f(z')$. Together with $f(z)\prec f(z')$, we have $f(x')\prec f(x)$, It completes the proof of Assertion II, and thus, also the proof of the lemma.
\end{proof}

\begin{Lem}
	If $n\geq2$, $a,b,c\in \R_+U$ are collinear in this order such that $o\notin\aff\{a,b,c\}$, then
	\begin{gather}\label{eq:03-11-1552}
		\lim_{\lambda\rightarrow 0+}|f(\lambda b)|\leq
		\max\{\lim_{\lambda\rightarrow 0+}|f(\lambda a)|,\lim_{\lambda\rightarrow 0+}|f(\lambda c)|\}.
	\end{gather}
	Further if  $\lim_{\lambda\rightarrow 0+}|f(\lambda a)|+\lim_{\lambda\rightarrow 0+}|f(\lambda c)|>0$ and $a,b,c$ are distinguishable, then
	\begin{gather}\label{eq:03-11-1553}
		\lim_{\lambda\rightarrow 0+}|f(\lambda b)|<
		\max\{\lim_{\lambda\rightarrow 0+}|f(\lambda a)|,\lim_{\lambda\rightarrow 0+}|f(\lambda c)|\}.
	\end{gather}
\end{Lem}

\begin{proof}
    Since $a, b, c$ are collinear in this order and $o\notin\aff\{a,b,c\}$, one has that $\lambda a$, $\lambda b$, $\lambda c$  are collinear in this order for all $\lambda>0$, and $o\notin\aff\{\lambda a,\lambda b,\lambda c\}$.
    
    Hence, Lemma \ref{Lem:Collinear} gives that $f(\lambda a),f(\lambda b),f(\lambda c)$ are collinear in this order. 
    Moreover, by Lemma \ref{Lem:OneRayMapstoOneRay} and Lemma \ref{Lem:AllMapstoAll}, $o\notin\aff\{f(\lambda a),f(\lambda b),f(\lambda c)\}$.
    
By Lemma \ref{Lem:fIncreasing}, one gets
\begin{gather*}
\lim_{\lambda\rightarrow 0+}f(\lambda a)\in \R_+f(a),~ \lim_{\lambda\rightarrow 0+}f(\lambda b)\in \R_+f(b),~ \lim_{\lambda\rightarrow 0+}f(\lambda c)\in \R_+f(c).
\end{gather*}
It is also clear that
$\lim_{\lambda\rightarrow 0+}f(\lambda a),~\lim_{\lambda\rightarrow 0+}f(\lambda b),~\lim_{\lambda\rightarrow 0+}f(\lambda c)$ are collinear in this order.
Thus \eqref{eq:03-11-1552} holds and the equality holds only if $\lim_{\lambda\rightarrow 0+}f(\lambda a) = \lim_{\lambda\rightarrow 0+}f(\lambda c)=o$ or $a,b,c$ are not distinguishable.
It completes the proof.
\end{proof}

\begin{Lem}\label{Lem:LimitesTo0ForN=2}
	If $n=2$, then $\lim_{\lambda\rightarrow0+}|f(\lambda x)|=0$ for all $x\in U$.
\end{Lem}

\begin{proof}
Set
\begin{gather}\label{eq:03-13-1451}
    \xi=\sup_{x\in U}\lim_{\lambda\rightarrow0+}|f(\lambda x)|.
\end{gather}
There are $\{x_k\}_{k=1}^\infty\subset U$ such that $\lim x_k=a\in\mathbb{S}^1$ and $\lim_{k\rightarrow+\infty}\lim_{\lambda\rightarrow0+}|f(\lambda x_k)|=\xi$.
Without loss of generality, we may assume that $\{x_k\}_{k=1}^\infty\subset\relint[R_b,R_c]$ with $b,c\in U$ and $b+c\neq o$.
Thus, there are two sequences $\{\mu_k\}_{k=1}^\infty$,  $\{\nu_k\}_{k=1}^\infty$ of positive real numbers such that $\mu_kb$, $x_k$ and $\nu_kc$ are collinear.
By \eqref{eq:03-11-1552}, one has
\begin{gather*}
	\lim_{\lambda\rightarrow0+}|f(\lambda x_k)|\leq\max\{\lim_{\lambda\rightarrow0+}|f(\lambda \mu_kb)|,\lim_{\lambda\rightarrow0+}|f(\lambda \nu_kc)|\}=\max\{\lim_{\lambda\rightarrow0+}|f(\lambda b)|,\lim_{\lambda\rightarrow0+}|f(\lambda c)|\}.
\end{gather*}
Therefore $\xi<+\infty$.

Now	suppose $\xi>0$.
	Thus, $\lim_{\lambda\rightarrow0+}|f(\lambda b_1)|>0$ for some $b_1\in U$. By the density of $U$ in $\mathbb S^1$, we may
	choose different points $b_2,b_3\in U$, which are different from $b_1$, such that $o\in\relint[b_1,b_2,b_3]$  and
	\begin{gather}\label{eq:03-13-1427}
	    \lim_{\lambda\rightarrow0+}|f(\lambda b_1)|+\lim_{\lambda\rightarrow0+}|f(\lambda b_2)|+\lim_{\lambda\rightarrow0+}|f(\lambda b_3)|>0.
	\end{gather}
	Further choose different points $c_1,c_2,c_3\in 
	\R_+U$ such that $o\in\relint[c_1,c_2,c_3]$, $c_1\in\relint[b_2,b_3]$, $c_2\in\relint[b_3,b_1]$, $c_3\in\relint[b_1,b_2]$.
	
	We claim
	\begin{align}
	    \lim_{\lambda\rightarrow 0+}|f(\lambda x)|&\leq \max\{\lim_{\lambda\rightarrow 0+}|f(\lambda c_1)|,\lim_{\lambda\rightarrow 0+}|f(\lambda c_2)|,\lim_{\lambda\rightarrow 0+}|f(\lambda c_3)|\}\label{eq:03-13-1357}\\
	    &<\max\{\lim_{\lambda\rightarrow 0+}|f(\lambda b_1)|,\lim_{\lambda\rightarrow0+}|f(\lambda b_2)|,\lim_{\lambda\rightarrow 0+}|f(\lambda b_3)|\}\label{eq:03-13-1447}
	\end{align}
	for all $x \in U$.
	Set $\xi_0:=\max\{\lim_{\lambda\rightarrow 0+}|f(\lambda c_1)|,\lim_{\lambda\rightarrow 0+}|f(\lambda c_2)|,\lim_{\lambda\rightarrow 0+}|f(\lambda c_3)|$.
	By \eqref{eq:03-13-1451}, then it is  clear that  
\begin{gather*}
\lim_{\lambda\rightarrow 0+}|f(\lambda x)| 
\leq \xi_0
<\xi  
\end{gather*}
	for all $x\in U$,
	which contradicts the definition of $\xi$. Hence, $\xi=0$.
	
The last part of the proof is to show that \eqref{eq:03-13-1357} and \eqref{eq:03-13-1447} hold.
In fact, we can assume w.l.o.g, $x\in [R_{c_2},R_{c_3}]$. Let $d=R_x\cap[c_2,c_3]$. Remember \eqref{eq:03-11-1552}, which deduces that
\begin{gather*}
\lim_{\lambda\rightarrow 0+}|f(\lambda x)|=\lim_{\lambda\rightarrow 0+}|f(\lambda d)|\leq\max\{\lim_{\lambda\rightarrow 0+}|f(\lambda c_2)|,\lim_{\lambda\rightarrow 0+}|f(\lambda c_3)|\}.
\end{gather*}
It follows that \eqref{eq:03-13-1357} holds.
	
	
We can further assume w.l.o.g., 
\begin{gather}\label{eq:03-13-1446}
\lim_{\lambda\rightarrow 0+}|f(\lambda b_1)|=\max\{\lim_{\lambda\rightarrow 0+}|f(\lambda b_1)|,\lim_{\lambda\rightarrow0+}|f(\lambda b_2)|,\lim_{\lambda\rightarrow 0+}|f(\lambda b_3)|\}.
\end{gather}
Then, by \eqref{eq:03-13-1427}, $\lim_{\lambda\rightarrow 0+}|f(\lambda b_1)|>0$.
Since \eqref{eq:03-11-1552} and \eqref{eq:03-11-1553}, it follows that
\begin{gather}
\lim_{\lambda\rightarrow 0+}|f(\lambda c_2)|<\max\{\lim_{\lambda\rightarrow0+}|f(\lambda b_1)|,\lim_{\lambda\rightarrow 0+}|f(\lambda b_3)|\}\leq\lim_{\lambda\rightarrow 0+}|f(\lambda b_1)|,\label{eq:03-13-1434}\\
\lim_{\lambda\rightarrow 0+}|f(\lambda c_3)|<\max\{\lim_{\lambda\rightarrow0+}|f(\lambda b_1)|,\lim_{\lambda\rightarrow 0+}|f(\lambda b_2)|\}\leq\lim_{\lambda\rightarrow 0+}|f(\lambda b_1)|,\label{eq:03-13-1435} \\
\lim_{\lambda\rightarrow 0+}|f(\lambda c_1)|\leq\max\{\lim_{\lambda\rightarrow0+}|f(\lambda b_2)|,\lim_{\lambda\rightarrow 0+}|f(\lambda b_3)|\}\leq\lim_{\lambda\rightarrow 0+}|f(\lambda b_1)|.\label{eq:03-13-1433}
	\end{gather}
	
If $\lim_{\lambda\rightarrow0+}|f(\lambda b_2)|+\lim_{\lambda\rightarrow 0+}|f(\lambda b_3)|>0$, by \eqref{eq:03-11-1553}, then the first inequality in \eqref{eq:03-13-1433} is strict; if $\lim_{\lambda\rightarrow0+}|f(\lambda b_2)|+\lim_{\lambda\rightarrow 0+}|f(\lambda b_3)|=0$, then the second inequality in \eqref{eq:03-13-1433} is strict. Hence, one gets
\begin{gather}
\lim_{\lambda\rightarrow 0+}|f(\lambda c_1)|<\lim_{\lambda\rightarrow 0+}|f(\lambda b_1)|.\label{eq:03-13-1445}
\end{gather}
	
	From \eqref{eq:03-13-1446}, \eqref{eq:03-13-1434}, \eqref{eq:03-13-1435} and \eqref{eq:03-13-1445}, we have \eqref{eq:03-13-1447}. 
 It completes the proof.
\end{proof}	

\begin{Lem}\label{Lem:Point2PointForNEq2}
	If $n\geq2$, then $U=\mathbb{S}^{n-1}$, moreover, $\R_+U=\R^n$.
\end{Lem}

\begin{proof}
	For $n\geq3$, by \eqref{eq:03-11-1300}, it is clear that $U=\mathbb{S}^{n-1}$, and thus, $\R_+U=\R^n$.
	
	Now, assume that $n=2$. We prove it by contradiction. Suppose that $b\in\mathbb{S}^1\setminus U$ such that $\dim\varphi(\overline {\mu_0 b})=2$ for some $\mu_0>0$.	
	Choose a segment $[x_1,x_2]\subset\varphi(\overline {\mu_0 b})$ satisfying $R_{x_1}\cap R_{x_2}=\emptyset$, and choose $b_1,b_2\in U$ such that $b_1+b_2\neq o$ and $b\in\relint[R_{b_1},R_{b_2}]$.
	
	It is clear that $[x_1,x_2]\subset\varphi(\overline{\mu_0b})\subset\overline{f(\mu_0b_1)}\lor\overline{f(\mu_0b_2)}$ and $\varphi(\overline{\mu_0b})\cap(\overline{f(\mu_0b_1)}\cup\overline{f(\mu_0b_2)})=\emptyset$. Since $R_{x_1}\cap R_{x_2}=\emptyset$, we get $R_{x_1}$$f(\mu_0b_1)=\alpha_1x_1+\alpha_2x_2$ and $f(\mu_0b_2)=\beta_1x_1+\beta_2x_2$ for some $\alpha_1, \alpha_2,\beta_1,\beta_2\in\mathbb R$. Without loss of generality, let $\alpha_1,\beta_2>0$, $\alpha_2,\beta_1<0$. It is straightforward to show that
	\[[\gamma f(\mu_0b_1), x_2]\cap[\gamma f(\mu_0b_2), x_1]=\left\{\frac{(\gamma\beta_1+\gamma\beta_2-1)\gamma\alpha_1x_1+(\gamma\alpha_1+\gamma\alpha_2-1)\gamma\beta_2x_2}{\gamma^2(\alpha_1\beta_2-\alpha_2\beta_1)+\gamma(\alpha_2+\beta_1)-1}\right\}=:\{y_\gamma\}\]
	for small enough $\gamma>0$. Clearly, $\lim_{\gamma\to 0+} y_\gamma=o$. By Lemma \ref{Lem:LimitesTo0ForN=2}, we have $[\gamma f(\mu_0b_1), x_2]\cap[\gamma f(\mu_0b_2), x_1]\subset[\overline{f(\lambda b_1)},x_1,x_2]\cap[\overline{f(\lambda b_2)},x_1,x_2]$ for small enough $\lambda>0$. Observe that
	\begin{gather*}
		\begin{split}
			[\overline{f(\lambda b_1)},x_1,x_2]\cap[\overline{f(\lambda b_2)},x_1,x_2]&\subset \left(\overline{f(\lambda b_1)}\lor\varphi(\overline {\mu_0 b})\right) \land \left(\overline{f(\lambda b_2)}\lor\varphi(\overline {\mu_0 b})\right)\\
			&=\varphi\left((\overline{\lambda b_1}\lor\overline {\mu_0 b})\land(\overline{\lambda b_2}\lor\overline {\mu_0 b})\right)=\varphi(\overline {\mu_0 b}).
		\end{split}
	\end{gather*}
	Hence, $y_\gamma\in\varphi(\overline {\mu_0 b})$ for small enough $\gamma>0$. Together with the closedness of $\varphi(\overline {\mu_0 b})$, we have $o\in\varphi(\overline {\mu_0 b})$, which contradicts Lemma \ref{Lem:PhiPSIsNotBarO}. It completes the proof.
\end{proof}	

\begin{Lem}\label{Lem:fIsLinearity}
	If $n\geq2$, then $f\in\mathrm{GL}(n)$.
\end{Lem}
	
\begin{proof}
By Lemma \ref{Lem:OneRayMapstoOneRay}, $f(\R^n)$ is not contained in a line.
Together with $f(o)=o$, Lemmas \ref{Lem:Affinity},  \ref{Lem:Collinear} and \ref{Lem:Point2PointForNEq2}, it is enough to show that $f$ is injective.
		
Suppose $f(a')=f(a)$ but $a'\neq a$. By Lemma \ref{Lem:OneRayMapstoOneRay}, we can assume $a'\prec a\neq o$. 
Following from Lemma \ref{Lem:AlwaysRegular}, we can choose $b'\prec b$ with $f(b')\prec f(b)$ such that $a,b$ are linearly independent. 
Clearly there is $c\in\relint[a,b']\cap\relint[a',b]$.
By Lemma \ref{Lem:Collinear} and Lemma \ref{Lem:Point2PointForNEq2},
$f(a),f(c),f(b')$ are collinear, as well as $f(a'),f(c),f(b)$. 
Thus $f(a')=f(a)$ gives $f(b)=f(b')$, a contradiction. It completes the proof.
\end{proof}

Now Theorem \ref{Thm:MainThm_Identity} follows from the following.
\begin{Thm}\label{Thm:EndForNgeq3}
	If $n\geq2$, $\varphi$ is a non-constant endomorphism of $(\PC,\lor,\land)$, then $\varphi\in GL(n)$.
\end{Thm}
	
\begin{proof}
By Lemma \ref{Lem:fIsLinearity}, $f\in GL(n)$. 
We only need to verify $\varphi(A)=f(A)$ for all $A\in\PC$. If $A=\emptyset$ or $\R^n$, by Lemma \ref{Lem:EmptysetMapstoEmptyset} and Lemma \ref{Lem:AllMapstoAll}, there is nothing to prove. 
Now, let $A$ be a closed pseudo-cone.
	
	Since
	\begin{gather*}
		f(A)=\bigcup_{a\in A}f(a)\subset \bigcup_{a\in A}\overline{f(a)} = \bigcup_{a\in A}\varphi(\overline a)=\varphi(A),
	\end{gather*}
	it is sufficient to show that $\varphi(A)\subset f(A)$. If not, assume $a\in\varphi(A)\setminus f(A)$.
		
	Observe that
	\begin{gather*}
	    \varphi\left(\overline{f^{-1}(a)}\land A\right)=\varphi\left(\overline{f^{-1}(a)}\right)\land\varphi(A)=\overline{f(f^{-1}(a))}\land\varphi(A)=\overline a\land\varphi(A)=\overline a.
	\end{gather*}
	Thus, by Lemma \ref{Lem:EmptysetMapstoEmptyset}, $\overline{f^{-1}(a)}\land A\neq\emptyset$. Since $a\notin f(A)$, i.e. $f^{-1}(a)\notin A$, then
	\begin{gather*}
	    \overline{f^{-1}(a)}\land A=\overline{\lambda_0 f^{-1}(a)}\text{ for some }\lambda_0 > 1,
	\end{gather*}
	which implies that
	\begin{gather*}
		\overline a=\varphi\left(\overline{f^{-1}(a)}\land A\right)=\varphi\left(\overline{\lambda_0 f^{-1}(a)}\right)=\overline{f(\lambda_0 f^{-1}(a))}=\overline{\lambda_0 a},
	\end{gather*}
	a contradiction. It completes the proof.
\end{proof}
	
\subsection{Endomorphisms for \texorpdfstring{$n=1$}{n=1}}
Finally, we consider the case $n=1$ for completeness.
	
\begin{Thm}\label{Thm:EndForN=1}
	If $\varphi$ is an endomorphism of $(\mathcal{PC}^1,\lor,\land)$, then one of the following is true.
	
	(i)
	\begin{gather}\label{eq:03-11-2231}
	    \varphi(\overline x)=
	    \begin{cases}
	        K, & x>0,\\
	        L, & x<0,
	    \end{cases}
	\end{gather}
	for some $K,L\in\mathcal{PC}^1$, where $\varphi(\emptyset) = K \subset L = \varphi(\overline o)$ or $\varphi(\emptyset) = L \subset K = \varphi(\overline o)$.
	
	(ii) There is a monotone function $f:\R \to \R$ such that
	\begin{gather}
	    \varphi(\overline x)=\overline{f(x)}
	\end{gather}
	with $\varphi(\emptyset)=\emptyset$ and
	\begin{gather}\label{eq:03-13-1643}
	    f(x)=0\Longleftrightarrow x=0.
	\end{gather}
\end{Thm}
	
\begin{proof}
    Assume w.l.o.g that $\varphi$ is not constant. By Lemma \ref{Lem:const}, one gets $\varphi(\emptyset)\subsetneq\varphi(\overline 0)$.
    
	For $xy> 0$, we have
	\begin{gather}
			\varphi(\overline x)\lor \varphi(\overline {-y}) = \varphi(\overline 0), \label{eq:03-11-2239}\\
			\varphi(\overline x)\land \varphi(\overline {-y}) = \varphi(\emptyset).  \label{eq:03-17-1317}
	\end{gather}

    \myvskip
	\noindent\textit{Case I. } $\varphi(\overline 0)\neq\overline 0$.
    \myvskip	
	If $\varphi(\overline 0)=\emptyset$, then $\varphi$ is constant. Thus, $\varphi(\overline 0)\neq\emptyset$.
	
	Set $\overline a=\varphi(\overline 0)$ with $a\neq 0$. 
	For $x\neq 0$, by \eqref{eq:03-11-2239} with $y=x$, one gets $\varphi(\overline x), \varphi(\overline {-x}) \subset \overline a$.
	Together with \eqref{eq:03-17-1317}, $\varphi(\overline x)=\overline a$ and $\varphi(\overline {-x})=\varphi(\emptyset)$; or $\varphi(\overline x)=\varphi(\emptyset)$ and $\varphi(\overline {-x})=\overline a$.
	
	If $xy>0$, then, by \eqref{eq:03-11-2239}, \eqref{eq:03-17-1317} and the above result, it is clear that $\varphi(\overline x)=\varphi(\overline y)$ and $\varphi(\overline {-x})=\varphi(\overline {-y})$.
	
	Thus, \eqref{eq:03-11-2231} holds in Case I.
	
	\myvskip
	\noindent\textit{Case II. } $\varphi(\overline 0)=\overline 0$ and $\varphi(\emptyset)\neq\emptyset$.
    \myvskip			
	For $x\neq 0$, by \eqref{eq:03-17-1317} with $y=x$, one has $\varphi(\overline x), \varphi(\overline{-x})\neq\emptyset$. Set $\overline a=\varphi(\overline x)$ and $\overline b=\varphi(\overline{-x})$. Suppose $ab>0$. 
	It follows that $\overline a\lor\overline b\neq\overline 0$, a contradiction to \eqref{eq:03-11-2239}.
	Suppose $ab<0$. It follows that $\overline a\land\overline b=\emptyset$, a contradiction to \eqref{eq:03-17-1317}.
	Hence, $ab=0$. By $\overline a\land \overline b=\varphi(\emptyset)$, one gets $a\neq 0$ or $b\neq 0$. 
	That is, one of $\varphi(\overline x)$ and $\varphi(\overline{-x})$ is $\overline 0$, and the other is $\varphi(\emptyset)$.
	
	If $xy>0$, then, by \eqref{eq:03-11-2239}, \eqref{eq:03-17-1317} and the above result, it is clear that $\varphi(\overline x)=\varphi(\overline y)$ and $\varphi(\overline {-x})=\varphi(\overline {-y})$.
	
	Thus, \eqref{eq:03-11-2231} holds in Case II.	
	\myvskip	
	\noindent\textit{Case III. } $\varphi(\overline 0)=\overline 0$ and $\varphi(\emptyset)=\emptyset$.
	\myvskip		
	If $\{\varphi(\overline z),\varphi(\overline {-z})\}=\{\overline 0,\emptyset\}$ for some $z\neq 0$, then, by \eqref{eq:03-11-2239} and \eqref{eq:03-17-1317}, we get that $\varphi(\overline x)=\varphi(\overline z)$ and $\varphi(\overline{-x} )=\varphi(\overline{-z})$ for any $xz>0$.
	It deduces \eqref{eq:03-11-2231}.
	
	Otherwise, by \eqref{eq:03-11-2239} and \eqref{eq:03-17-1317}, $\emptyset\subsetneq\varphi(\overline x)\subsetneq\overline 0$ for all $x\neq 0$.
	Thus, one defines a function $$f:\R\rightarrow\R\text{ such that } \varphi(\overline x)=\overline{f(x)}.$$
	It is clear that
	\eqref{eq:03-13-1643} holds in Case III.

	If $f(1)>0$, then $f(x)>0$ for all $x>0$ since either $\overline{f(x)} \subset \overline{f(1)}$ or $\overline{f(x)} \supset \overline{f(1)}$.
	By $\varphi(\overline {-1})\land \varphi(\overline 1)=\emptyset$, we have $f(-1)<0$ which implies $f(x)<0$ for all $x <0$.
	For $x>y>0$ or $x<y<0$, one has $\overline x\subset \overline y$, which implies $f(x)>f(y)>0$ or $f(x)<f(y)<0$, respectively. Hence, $f$ is increasing.
	Similarly, $f$ is decreasing if $f(1)<0$.
\end{proof}

\section*{Acknowledgement}
\addcontentsline{toc}{section}{Acknowledgement}
The authors express sincere thanks to the reviewers for their careful reading the first and second versions of this paper, pointing out some language errors and for their valuable suggestions
and comments which improved the paper.
The work of the first and the third author was supported in part by the National Natural Science Foundation
of China (Project 12171304).

\bibliographystyle{crelle}

\end{document}